%% file: AI.tex
\RequirePackage{amsmath}

\documentclass[smallextended,envcountsect]{svjour3} 
\smartqed 
\usepackage{graphicx}
\journalname{JOTA}

\usepackage{fullpage}

\usepackage{amssymb,amscd,amsfonts,amstext,bbm, enumerate, dsfont, mathtools}
\usepackage{array,multirow}
\usepackage{xcolor}

\usepackage{tikz}
\usepackage{pgfplots}
\usetikzlibrary{plotmarks,patterns}
\usetikzlibrary{matrix}
\usepgfplotslibrary{groupplots}
\pgfplotsset{compat=newest}

\usepackage{url}

\usepackage{algorithm}
\usepackage[noend]{algpseudocode}

\usepackage[utf8]{inputenc}

\usepackage{hyperref}
\usepackage{cite}

\newtheorem{coro}[theorem]{Corollary}
\newtheorem{prop}[theorem]{Proposition}
\newtheorem{rmk}[theorem]{Remark}
\newtheorem{hyp}{Assumption}

\DeclareMathOperator*{\argmin}{argmin}
\DeclareMathOperator*{\prox}{\mathbf{prox}}



\newcommand{\RR}{\mathbb{R}}

\newcommand{\FF}{\Phi} 
\DeclareMathOperator{\dist}{dist}
\newcommand{\ip}[2]{\langle #1, #2 \rangle}

\begin{document}

\title{On the Proximal Gradient Algorithm with Alternated Inertia}


\author{Franck Iutzeler   \and  J\'er\^ome Malick }

\institute{Franck Iutzeler,  Corresponding author  \at
              Univ. Grenoble Alpes \\
              franck.iutzeler@univ-grenoble-alpes.fr
              \and
              J\'er\^ome Malick \at
              CNRS, LJK\\
              jerome.malick@univ-grenoble-alpes.fr
}

\date{Received: date / Accepted: date}

\maketitle

\begin{abstract}
In this paper, we investigate the attractive properties of the proximal gradient algorithm with inertia. Notably, we show that using {\em alternated inertia} yields monotonically decreasing functional values, which contrasts with usual accelerated proximal gradient methods. We also provide convergence rates for the algorithm with alternated inertia based on local geometric properties of the objective function. The results are put into perspective by discussions on several extensions and illustrations on common regularized problems.
\end{abstract}


\section{Introduction}
\label{sec:intro}


In this paper, we consider the composite convex optimization problem
\begin{equation}
\label{eq:pb}
\min_{x\in\mathbb{R}^n} F(x) :=  f(x) + g(x)
\end{equation}
where the functions $f$ and $g$ are convex, and $f$ is furthermore smooth. This problem can be solved by the proximal gradient algorithm, a special case of Forward-Backward splitting, the iterations of which read
\begin{align}
\label{eq:pg}
x_{k+1} &= \prox_{\gamma_k g} \left( x_k  - \gamma_k \nabla f(x_k) \right).
\end{align}
Recall that the proximal operator $\prox_h\colon\RR^n\rightarrow\RR^n$ of a proper lower semi-continuous convex function $h:\mathbb{R}^n\to\mathbb{R}\cup\{+\infty\}$ is defined as
$$ \prox_h(x) :=  \argmin_{w\in\mathbb{R}^n} \left( h(w) + \frac{1}{2} \left\| w-x \right\|^2  \right)\qquad \text{for all $x\in \mathbb{R}^n$} .$$
The computation of this operator is simple (closed-form or equivalent to a smaller problem) for several interesting functions such as the $\ell_1$ norm, the group norm $\ell_{1,2}$, or the nuclear norm, which are commonly used as regularizers in machine learning and image processing optimization problems; see, e.g., \cite{chambolle1998nonlinear,daubechies2004iterative,hale2008fixed}.  

To improve the convergence of the proximal gradient algorithm, \emph{inertial versions} have been extremely popular, both practically and in terms of theoretical rate \cite{alvarez2004weak,lorenz2014inertial,chambolle2015convergence,aujol2015stability,attouch2016rate}. Adding inertia consists in constructing the next iterate $x_{k+1}$ by combining the outputs $y_{k+1}$ and $y_k$ of the last two proximal steps. Specifically, given a sequence of real non-negative numbers $(\alpha_k)$, an iteration of the inertial proximal gradient algorithm can be written as
\begin{equation}\label{eq:ipg}
 y_{k+1}  =  \prox_{\gamma_k g} \left( x_k  - \gamma_k \nabla f(x_k) \right) ~, ~~~ x_{k+1}    =  y_{k+1} + \alpha_{k+1} ( y_{k+1} - y_k) .
\end{equation}
The design of the inertial sequence $(\alpha_k)$ affects greatly the performance of the resulting algorithm, and different options are investigated in the literature.
Popular choices include Nesterov's optimal sequence \cite{nesterov1983method,guler1992new} which leads to the FISTA algorithm \cite{beck2009fast} or a variant used for proving the iterates convergence \cite{chambolle2015convergence,attouch2016rate}. These sequences have the form

\begin{equation}\label{eq:nesta}
\alpha_{k+1} = \frac{t_k -1}{t_{k+1}}  ~~~~~~~~\text{ with } t_0 = 0 \text{ and } ~~~~
\begin{array}{rl}
  &   t_{k+1} = \displaystyle\frac{1+\sqrt{1+4t_k^2}}{2} \\[2ex]
~~~ \text{ or ~ } ~~~  &  t_{k+1} = \displaystyle \frac{k+a}{a} ~~~~~~ \text{ with } a>2.\\[2ex]
\end{array}
\end{equation}
%

Both choices lead to increasing sequences going to $1$ at rate $1/k$ and are proven to accelerate the worst-case convergence rate of the algorithm from $\mathcal{O}(1/k)$ to $\mathcal{O}(1/k^2)$. Another variant with $t_{k+1} = ((k+a)/a)^d$ where $d\in]0,1]$ and $a>\max\{1,(2d)^{\frac{1}{d}} \}$ was shown to be efficient, when the proximal operator is computed approximately \cite{aujol2015stability}. Finally, if $f$ is in addition $\mu$-strongly convex, then the optimal linear rate is attained for fixed $\alpha = (1-\sqrt{\mu/L})/(1+\sqrt{\mu/L}) $ ~\cite{nesterov2005smooth}.

\paragraph{Quest for Monotonicity.}

An attractive feature of the vanilla proximal gradient algorithm is its monotonicity; more precisely, iterates generated by \eqref{eq:pg} satisfy both
\begin{itemize}
\item monotonicity of the iterates in the sense of Fej\'{e}r: $\|x_{k+1}-x\| \leq \|x_{k}-x\|$ for any optimal $x$

\smallskip

\item monotonicity of the functional values: $F(x_{k+1})\leq F(x_k)$. 
\end{itemize}
These properties are well-known and obtained directly from contraction and descent results recalled in Section~\ref{sec:pre}. On the contrary, acceleration by inertia breaks down monotonicity as the iterates generated by inertial variants can circle or oscillate around the set of minimizers \cite{mainge2008convergence,beck2009fastb,lorenz2014inertial}. These kinds of behaviors make accelerated methods sometimes slower than their unaccelerated counterparts (see for instance the non-negative least squares problem in~\cite{malitsky2016first}).

Monotonicity (and in particular functional monotonicity) is a highly desirable feature in optimization, in both theory and practice; see for instance the quests for descent in \cite{correa1993,bioucas2007new,fuentes2012descentwise}. For composite optimization, several algorithms based on descent tests have been proposed to fix the non-monotonicity of accelerated proximal gradient methods, notably: MTwist \cite{bioucas2007new}, MFISTA \cite{beck2009fastb}, or Monotonous APG \cite{li2015accelerated}. However, all these methods rely on additional function evaluations, or even additional computations, compared to the initial proximal gradient algorithm.

Another strategy is to use inertia every other step of the method: the corresponding algorithm, called proximal gradient with {\em alternated} inertia, can be formulated as \eqref{eq:ipg} with $\alpha_k = 0$ for $k$ even, and $\alpha_k > 0$ for $k$ odd. Alternated inertia was introduced in \cite{mu2015note} as a variation of the proximal point algorithm, and exhibited attractive performances in practice \cite{iutzeler2016generic}; and it was shown to have potential to achieve good rates (better than both inertia and over-relaxation in some cases) while recovering iterates monotonicity for some range of inertial coefficients. Unfortunately, in the case of the proximal gradient, the range guaranteeing iterates monotonicity is typically $[0,0.5]$, which prevents from considering usual inertial sequences. Moreover, functional monotonicity has not been investigated yet.


\paragraph{Contributions and Outline.}

In this paper, we show that the alternated inertial proximal gradient algorithms enjoys a monotonic functional descent for a range of inertial coefficients encompassing the popular sequences of the literature. Furthermore, building on this monotonic decrease, we provide a complexity analysis based on the geometry of the function around optimal points. In addition, we discuss several extensions of the alternated inertial proximal algorithm for i) being resilient to undetected strong convexity, ii) handling non-convex $g$, and iii) recovering a $1/k^2$ worst-case rate. Finally, we provide numerical examples illustrating the attractive properties of alternated inertia compared to existing (accelerated) proximal gradient algorithms.

The outline of this paper is as follows. Section\;\ref{sec:pre} sets the framework with the assumptions and useful known results. Section\;\ref{sec:main} is the core of the paper and presents the study of the alternated inertial proximal gradient algorithm: its monotonicity and the resulting complexity analysis. This analysis is based on the standard Kurdyka-{\L}ojasiewicz inequality but in an original way: a by-product of our developments is indeed a refined analysis of functional convergence rates for general weakly decreasing sequences. Finally, Section\,\ref{sec:ext} discusses extensions and Section~\ref{sec:num} presents illustrative numerical experiments.

\section{Preliminaries: Assumptions and Recalls on the Proximal Gradient algorithm}
\label{sec:pre}

In this paper, the working space is $\RR^n$ with the standard inner product $\ip{\cdot}{\cdot}$ and associated Euclidean norm~$\|\cdot \|$. Although all forthcoming results generalize to Hilbert spaces, we restrict ourselves to finite-dimensional spaces for simplicity and for being closer to considered applications. We use in our developments standard terminology, notation, and tools from convex analysis and operator theory; we refer to the textbooks \cite{hiriart-lemarechal-1993} and \cite{bauschke2011convex}. In particular, for a convex function $h$, we denote by $\partial h (x)$ the subdifferential of $h$ at $x$.

We make the following standard assumptions on the functions involved in the composite problem~\eqref{eq:pb}.
\begin{hyp}
\label{hyp:gen}
\begin{itemize}
\item[(i)] $f:\mathbb{R}^n\to\mathbb{R}$ is a convex $L$-smooth function (for $L>0$) i.e.~it is differentiable and
$$\| \nabla f(x) - \nabla f(y) \| \leq L \| x-y\| \qquad \text{for all} ~ x,y\in\mathbb{R}^n; $$ 
\item[(ii)] $g:\mathbb{R}^n\to\mathbb{R}\cup\{+\infty\}$ is a {convex} and proper lower semi-continuous function:
\smallskip
\item[(iii)] the set of optimal points $X^\star := \argmin_x F(x) \neq \emptyset$ is closed and $F^\star := F(x^\star) > - \infty$ for any $x^\star\in X^\star$.
\end{itemize}
\end{hyp}

These three points will be assumed to hold for all the results stated in this paper (except for \emph{(ii)} in Section\;\ref{sec:noncvx} dealing with the non-convex case). For simplicity, we will also assume that the step-sizes of the algorithms are fixed over time: $\gamma_k = \gamma >0$ for all $k$; all the results presented in this paper extend easily to the case when step-sizes are bounded consistently with the Lipschitz constant $L$. 
Restricting to fixed step-sizes will slightly enlighten the statements and allow us to focus on the inertia parameters $(\alpha_k)$ which are at the heart of our analysis. 

We now define the proximal gradient operator for $F = f+g$ (for the step $\gamma>0$) by 
\begin{equation*}
\mathsf{T}_\gamma(x) := \prox_{\gamma g} \left( x  - \gamma \nabla f(x) \right).
\end{equation*}
An iteration of the proximal gradient algorithm \eqref{eq:pg} can thus be written compactly $x_{k+1} = \mathsf{T}_\gamma(x_k)$ and the set $X^\star$ of the optimal solutions of \eqref{eq:pb} coincide with the fixed points set of $\mathsf{T}_\gamma$.

In our analysis, we will rely on some basic results about this operator that we recall in the three next lemmas. We refer respectively to \cite[Th.\;25.8/Cor.\;27.9]{bauschke2011convex}, \cite[Lemma 1]{chambolle2015convergence}, and \cite[Prop.~13]{bolte2015error} for these results but out of convenience for curious readers, proofs of the last two lemmas are given in the appendix.


\begin{lemma}[Contraction]
\label{lem:op}
Take $\gamma\in]0,\frac{2}{L}[$. Then, the proximal gradient operator $\mathsf{T}_{\gamma}$ is $\nu$-averaged:
$$ \|\mathsf{T}_\gamma(x) - \mathsf{T}_\gamma(y)\|^2 +\frac{1-\nu}{\nu} \| (x - \mathsf{T}_\gamma (x)) - (y  - \mathsf{T}_\gamma(y) ) \|^2 \leq \|x-y\|^2 \qquad \text{\;for all $x,y\in \mathbb{R}^n$}
$$
with $\nu = 2/\big(1+2\min\{1,   1/(\gamma L)\}\big) \in[2/3,1[$.
\end{lemma}

\begin{lemma}[Functional descent]
\label{lem:1}
Take $\gamma>0$. Then, 
\begin{equation*}
F( \mathsf{T}_\gamma(x) ) + \frac{ (1-\gamma L)}{2\gamma}  \left\|  \mathsf{T}_\gamma(x) - x  \right\|^2   + \frac{1}{2\gamma} \left\| \mathsf{T}_\gamma(x) - y \right\|^2 \leq F(y)  + \frac{1}{2\gamma} \left\| x - y \right\|^2\qquad \text{\;for all $x,y\in \mathbb{R}^n$}.
\end{equation*}
\end{lemma}

\begin{lemma}[Bound on distance to subdifferential]
\label{lem:klopt}
Take $\gamma>0$. Then, 
\begin{equation*}
\dist(0,\partial F(\mathsf{T}_\gamma(x) )) \leq \frac{L\gamma + 1}{\gamma} \|x-\mathsf{T}_\gamma(x) \|\qquad \text{\;for all $x\in \mathbb{R}^n$}.
\end{equation*}
\end{lemma}



As a direct consequence of these lemmas, we retrieve the well-known monotonicity results for the proximal gradient algorithm.
For $\gamma\in]0,2/L[$ and any $x^\star \in X^\star$, we have indeed that the sequence generated by \eqref{eq:pg} satisfies
\begin{align*}
     \|x_{k+1} - x^\star\|^2 &\leq \|x_k-x^\star\|^2 - \frac{1-\nu}{\nu} \| x_{k+1} - x_k \|^2 \\
    F(x_{k+1}) &\leq F(x_k) - \frac{2-\gamma L}{2\gamma} \|x_{k+1}-x_k\|^2\\
    &\leq  F(x_k) - \frac{(2-\gamma L)\gamma}{(1+\gamma L)^2}\dist(0,\partial F(x_{k+1} ))^2. 
\end{align*}

Thus, $(x_k)$ converges to a point in $X^\star$ monotonically in both \emph{Fej\'{e}r} and \emph{functional} sense. As already mentioned in the introduction, inertial versions of the algorithm lose monotonicity, which is, roughly speaking, due to local changes in the shape of the function (we refer to the references given in the introduction and numerical experiments). 
We will see in the next section that alternating inertia allows to keep some monotonicity properties of the initial proximal gradient method.




\section{Proximal Gradient Algorithm with Alternated Inertia}\label{sec:main}

In this section, we study the proximal gradient algorithm with alternated inertia, introduced in \cite{mu2015note} and \cite{iutzeler2016generic}. With our notations, an iteration of this algorithm reads, for $k$ even, as
\begin{align}\label{eq:altin}
\left\{ \begin{array}{ll} y_{k+1} = \mathsf{T}_\gamma(x_k)  & ~~~ y_{k+2} = \mathsf{T}_\gamma(x_{k+1}) \\ x_{k+1} =  y_{k+1} + \alpha_k( y_{k+1} - y_{k} ) &  ~~~x_{k+2} = y_{k+2}  \end{array} \right. .
\end{align}
In Section \ref{sec:ite}, we briefly discuss known results about the iterates monotonicity for this method. We then derive in Section \ref{sec:func} a functional monotonicity result for a wide range of inertial parameters. Finally, in Section~\ref{sec:KL}, we leverage this monotonicity to provide convergence rates using geometrical properties of the objective function. This study is also a chance to establish an new generic convergence analysis using Kurdyka-{\L}ojasiewicz gradient inequality (see forthcoming Theorem~\ref{th:ratekl2}).

\subsection{Recalls about Iterates Monotonicity}
\label{sec:ite}

The iterates monotonicity of the proximal gradient with alternated inertia is established in \cite[Lemma~6]{iutzeler2016generic}.

\begin{theorem}[Iterates Monotonicity]
Take $\gamma\in]0,\frac{2}{L}[$. If the inertial sequence $(\alpha_k)$ verifies
\hspace*{-0.05cm}
$$ 0 \leq \alpha_k \leq \min\left\{1 ,  \frac{1}{\gamma L}\right\} -\frac{1}{2}  \qquad \text{for all $k>0$,}$$
\hspace*{-0.05cm}
then the sequence $(y_{2k})$ produced by \eqref{eq:altin} converges Fej\'{e}r monotonically to a point in $X^\star$.
\end{theorem}

For instance, with the typical choice $\gamma = 1/L$, this result states that taking a fixed inertia $\alpha_k = \alpha \in [0, 0.5]$ leads to a monotonic convergence. This is an interesting property which is verified for the vanilla proximal gradient algorithm but not by other accelerated proximal gradient algorithms; see e.g.\;\cite{mainge2008convergence}. However, this result suffers from the fact that monotonicity cannot be guaranteed when inertia goes to $1$, as for the popular Nesterov sequence \eqref{eq:nesta}.

\subsection{Functional Monotonicity}
\label{sec:func}

As recalled in Section~\ref{sec:pre}, the proximal gradient algorithm with $\gamma \in ]0,2/L[$ generates function values $(F(x_k))$ that decrease monotonically to $F^\star$. Interestingly, a monotonic behavior can be retrieved when considering alternated inertia.

\begin{theorem}[Descent result]
\label{th:main_des}
Take $\gamma>0$, $\alpha_k \geq 0$. Then, for an iteration of \eqref{eq:altin} with $k>0$ even,  
\begin{align}
\nonumber F( y_{k+2} )  &\leq F(y_k)  - \frac{(2-\alpha_k-\gamma L)}{2\gamma} \left[  \left\| y_{k+2} - x_{k+1} \right\|^2  + \left\| y_{k+1} - x_k \right\|^2 \right] \\
\label{eq:a_desKL} &\leq F(y_k)  - \frac{(2-\alpha_k-\gamma L)\gamma}{2(1+\gamma L)^2}  \dist(0,\partial F(y_{k+2} ))^2 .
\end{align}
\end{theorem}

\begin{proof}
Applying Lemma~\ref{lem:1} with $x = y = x_k = y_k$, we get
\begin{equation}
\label{eq:1}
F( y_{k+1} ) + \frac{(2 - \gamma L)}{2\gamma} \left\| y_{k+1} - x_k \right\|^2 \leq F(y_k)  
\end{equation}
and by applying Lemma~\ref{lem:1} with $x = x_{k+1}$ and $y = y_{k+1}$, we get {\small
\begin{equation}
\label{eq:2}
F( y_{k+2} ) + \frac{(1 - \gamma L)}{2\gamma} \left\| y_{k+2} - x_{k+1} \right\|^2 + \frac{1}{2\gamma} \left\| y_{k+2} - y_{k+1} \right\|^2 \leq F(y_{k+1})  + \frac{1}{2\gamma} \left\| x_{k+1} - y_{k+1} \right\|^2.
\end{equation}}
Summing Eqs.~\eqref{eq:1} and \eqref{eq:2}, we obtain
\begin{align*}
F( y_{k+2} ) &\leq  F(y_k)  +  \frac{1}{2\gamma} \left\| x_{k+1} - y_{k+1} \right\|^2 - \frac{1}{2\gamma} \left\| y_{k+2} - y_{k+1} \right\|^2   \\
&~~~~~ - \frac{(2-\gamma L)}{2\gamma} \left\| y_{k+1} - x_k \right\|^2 - \frac{(1 - \gamma L)}{2\gamma} \left\| y_{k+2} - x_{k+1} \right\|^2  \\
&=  F(y_k)  - \frac{1}{2\gamma} \left\| y_{k+2} - x_{k+1} \right\|^2   -  \frac{1}{\gamma} \langle y_{k+2} - x_{k+1}  ; x_{k+1}  - y_{k+1}  \rangle \\
&~~~~~ - \frac{(2-\gamma L)}{2\gamma} \left\| y_{k+1} - x_k \right\|^2 - \frac{(1 - \gamma L)}{2\gamma} \left\| y_{k+2} - x_{k+1} \right\|^2  \\
&=  F(y_k)  -  \frac{\alpha_k}{\gamma} \langle y_{k+2} - x_{k+1}  ; y_{k+1}  - y_{k}  \rangle\\
&~~~~~  - \frac{(2-\gamma L)}{2\gamma} \left\| y_{k+1} - x_k \right\|^2 - \frac{(2 - \gamma L)}{2\gamma} \left\| y_{k+2} - x_{k+1} \right\|^2  \\
%
%
&\leq  F(y_k)  - \frac{(2-\alpha_k-\gamma L)}{2\gamma} \left[  \left\| y_{k+2} - x_{k+1} \right\|^2  + \left\| y_{k+1} - x_k \right\|^2 \right]
\end{align*}
where the final inequality comes from Cauchy-Schwarz and Young's inequalities. Lemma~\ref{lem:klopt} then directly gives the second part of the result. \qed
\end{proof}

Instantiating this result with the standard step-size $\gamma = 1/L$, we have that the choice $\alpha_k = \alpha = 1- \varepsilon $ for any $\varepsilon\in]0,1]$ leads to a monotonic functional convergence. This is in contrast with Section\;\ref{sec:ite} where the condition for iterates monotonic convergence was $\alpha\leq 0.5$. Interestingly, if $(\alpha_k)$ is chosen as in \eqref{eq:nesta} or any sequence valued in $[0,1]$, the functional descent is preserved (although maybe not the convergence). Note also that, unlike MFISTA, MAPG, or MTwist, no functional evaluation is needed to guarantee this monotonicity. Thus, this monotonicity brought by alternating inertia has a practical interest; we will see in the next section that it also has a theoretical interest as it opens the door for a complexity analysis.

\subsection{Complexity Analysis}
\label{sec:KL}

The functional monotonicity of an algorithm can be combined with some geometric properties of the objective function in order to derive convergence rates. Two types of geometric profiles are often used: error bounds or Kurdyka-{\L}ojasiewicz gradient inequalities (see the extended survey in \cite{thesetrong} or the associated article \cite{bolte2015error}). For a nonsmooth\footnote{For a nonsmooth (possibly nonconvex) function $\FF\colon\RR^n\rightarrow\RR$, we denote by $\partial \FF(x)$ the limiting (Fr\'echet) subdifferential at $x$ \cite{RocWet98}. If $\FF$ is convex, this subdifferential coincides with the standard convex subdifferential.} function $\FF\colon\RR^n\rightarrow\RR$ achieving its minimum $\FF^\star$ (so that $\argmin \FF \neq \emptyset$),
these local properties write for $x \notin \argmin \FF$ as
\begin{itemize}
    \item \emph{error bounds} ~~~~~~~~~~~~~~~~~~~~~~~~~~ $\varphi(\FF(x)-\FF^\star)\geq \dist(x,\argmin \FF)$
    \item \emph{Kurdyka-{\L}ojasiewicz} (KL) ~~~~~~~~ $\varphi'(\FF(x)-\FF^\star) \dist(0,\partial \FF(x))\geq 1$ 
\end{itemize}
where $\varphi$ is a smooth increasing concave function, called \emph{desingularizing function} (see \cite{Bolte2007KL} or \cite{bolte2015error} for details). Typical desingularizing functions are of the form $\varphi(t) = C t^\theta /\theta$ for $\theta\in]0,1]$ and $C>0$. Interestingly, with this class of desingularizing functions, the two properties are equivalent for convex functions with the same desingularizing function \cite[Th.~5]{bolte2015error}. Also, in this case, we can reformulate the KL property in a simpler way: $\FF$ has the KL property (with $\theta\in]0,1]$ and $C>0$)  if 
\begin{align}\label{eq:klloc2}
\dist(0,\partial \FF(x)) \geq 
\frac{1}{C} ( \FF(x)-\FF^\star )^{1-\theta}   \qquad~~~ \text{for all $x \notin \argmin \FF$  }.
\end{align}
The KL property may look strong at a first sight but it turns out to be widely satisfied. In particular, any convex semi-algebraic function verifies the KL property \cite{Bolte2007KL}. In some special cases, we know explicitly the parameters $\theta$ and $C$ of the desingularizing function; see the examples in \cite[Sec.\,3]{bolte2015error}. For instance, we know that $\ell_1$-regularized least-squares functions
$$
\FF(x) = \frac{1}{2} \left\|Ax-b\right\|_2^2 +  \lambda_1 \|x\|_1
$$
have the KL property \eqref{eq:klloc2} on $\ell_1$-balls with $\varphi(s)= C\sqrt{s}$ with an explicit constant $C$ from $A$, $b$, $\lambda_1$ and the size of the ball \cite[Lem.\;10]{bolte2015error}.

The KL property is particularly well-suited for complexity analysis
of descent algorithms: see in particular \cite{attouch2009convergence,frankel2015splitting} for so-called subgradient descent algorithms,
or \cite{li2015accelerated} for algorithms showing a descent property of the form 
$$ \FF(x_{k+1}) \leq \FF(x_k) - a ~ \dist^2(0,\partial \FF(x_{k+1})).  $$

However, Theorem~\ref{th:main_des} does not offer such properties for typical $\gamma=1/L$ and Nesterov-like inertial sequences:
the decreasing term in \eqref{eq:a_desKL} may go to $0$ but not faster than $1/k$ which we denote\footnote{$a_k = \Omega(b_k) $ if $\exists a,K$ such that $\forall k\geq K$ we have $a_k\geq a.b_k$  } $a_k = \Omega(1/k)$. We thus provide an extension of existing results to cover our situation: the following theorem states functional convergence rates for general algorithms showing the weak descent property \eqref{eq:weak}, as the one of alternated inertia proximal gradient algorithm.

\begin{theorem}[General convergence with weak descent]
\label{th:ratekl2}
Let $\FF$ verify the KL property \eqref{eq:klloc2}. Suppose that there are a non-negative sequence $(a_k)$ and an $\mathbb{R}^n$-valued sequence $(x_k)$ verifying 
\begin{align}
    \label{eq:weak}
\FF(x_{k+1}) \leq \FF(x_k)  - a_k  \dist^2(0,\partial \FF(x_{k+1})) ~~~ \text{ and } ~~~ \sum_{k=1}^\infty a_k = +\infty 
\end{align}
then $\FF(x_k)$ converges to $\FF^\star$ and 
\begin{align*}
    \FF(x_{k+1}) - \FF^\star = \left\{ \begin{array}{ll}
       \mathcal{O}\left( \frac{1}{\left(\sum_{\ell=0}^k \min\{a_\ell,C_l\}\right)^{\frac{1}{1-\theta}}} \right)  & \text{ for } \theta\in]0,0.5[  \\[0.4cm]
      \mathcal{O}\left( \exp\left( -\frac{1}{2C^2} \sum_{\ell=0}^k \min\{a_\ell , C_e \} \right) \right)    &  \text{ for } \theta \in[0.5,1[ \\[0.4cm]
        0 ~~ \forall k\geq K := \inf\left\{ k :  \sum_{\ell=0}^k a_\ell > C^2 (\FF(x_0) - \FF^\star)  \right\}        &  \text{ for } \theta = 1 \\
    \end{array} \right .
\end{align*}
for positive constants $C_l, C_e$. Then, depending on the behavior of $(a_k)$ divided in three regimes:
\begin{itemize}
\item[(a)] $a_k \geq a >0$
\item[(b)] $a_k \geq 0$ with $a_k = \Omega\left( \frac{1}{k^d}\right)$, $d\in]0,1[$
\item[(c)] $a_k \geq 0$ with $a_k = \Omega\left( \frac{1}{k}\right)$
\end{itemize}
we get the functional convergence rates displayed in Table~\ref{tab:rates}.
\end{theorem}

\begin{table}[!ht]
\begin{tabular}{c|c|c|c|}
& $\theta \in]0,0.5[$ & $\theta \in[0.5,1[$ & $\theta = 1$ \\
\hline
(a) & $\mathcal{O}\left(\frac{1}{k^{1+\frac{2\theta}{1-2\theta}}}\right)$ & $\mathcal{O}\left( \left[\frac{C^2}{ C^2 + a  }\right]^k \right)$ & finite \\
\hline
(b) & $\mathcal{O}\left( \frac{1}{k^{1 + \frac{2\theta-d}{1-2\theta}}} \right) $ & $\mathcal{O}\left( \exp\left(-\frac{C'}{2C^2}k^{1-d}\right) \right)$ &  finite \\
\hline
(c) & $\mathcal{O}\left( \frac{1}{\log(k)^{\frac{1}{1-2\theta}}} \right)$ & $\mathcal{O}\left( \frac{1}{k^{\frac{C"}{2C^2}}} \right)$ & finite  \\
\hline
\end{tabular}
\caption{Functional convergence rates with $C' =  \lim\inf_k s_k/k^{1-d} $, $C" = \lim\inf_k s_k/\log(k)$ and $s_k = \sum_{\ell\leq k} \min(a_\ell,2C^2) $. }
\label{tab:rates}
\end{table}

\begin{proof} 
The monotonicity of $(\FF(x_k))$ implies that $\FF(x_k) \to \bar{\FF}$. Since $\sum_{k=1}^\infty a_k = +\infty $, we have moreover that a subsequence of $(\dist(0,\partial \FF(x_{k})))$ vanishes. By using \eqref{eq:klloc2}, this yields that the corresponding subsequence of $(\FF(x_k))$ converges to $\FF^\star$, hence $\bar{\FF} = \FF^\star$.

Define $r_k := \FF(x_k)-\FF^\star$, then $r_k\to 0$ monotonically. Thus, one can again use the KL property \eqref{eq:klloc2} and the descent on $\FF(x_k)$ to get
\begin{align}
\label{eq:basis2}
r_{k+1} \leq r_k - \frac{a_k}{C^2}  r_{k+1}^{2-2\theta} ~~~ \Leftrightarrow ~~~  \frac{a_k}{C^2}  \leq (r_k - r_{k+1}) r_{k+1}^{2\theta-2}.
\end{align}
which will be our core equation for rates derivations, which we separate in three cases depending on $\theta$.

\smallskip
\noindent\underline{Case $\theta= 1$.} Eq.~\eqref{eq:basis2} becomes $r_{k+1} \leq r_k - \frac{a_k}{C^2}  $, so by summing this inequality we get
$$ r_{k+1} \leq  r_0 - \frac{1}{C^2} \sum_{\ell=0}^k a_\ell$$
and as $(a_k)$ is a non-negative sequence verifying $\sum_{k=1}^\infty a_k = +\infty$, there is $K<\infty$, such that $\forall k\geq K$, $\frac{1}{C^2} \sum_{\ell=0}^k a_k > r_0$ leading to $ r_{k+1} =   \FF(x_{k+1})-\FF^\star < 0$ which contradicts $\FF^\star$ being the minimum of $\FF$. Thus, we must have $\FF(x_{k+1}) = \FF^\star$ for  all $k\geq K$, i.e. finite convergence.

\medskip
\noindent\underline{Case $\theta\in[0.5,1[$.} Here $0<2-2\theta\leq1$. Since $r_k \to 0$ monotonically, we have $r_k^{2-2\theta} \geq r_k $ for all $k\geq K$,
$$ 
r_{k+1} \leq r_k -  \frac{a_k}{C^2}   r_{k+1} \Leftrightarrow r_{k+1} \leq \frac{1}{1+  \frac{a_k}{C^2}  } r_{k}  
$$
which leads to different convergence modes depending of $(a_k)$. If it is bounded away from zero, linear convergence arises (case \emph{(a)}). Else, one gets that the above inequality also holds with $a_k$ replaced by $a'_k = \min(a_k, 2C^2)$, and we have $\log ( 1/(1+  \frac{a'_k}{C^2}  ) ) \leq -  \frac{a'_k}{2C^2}$. In addition, if $a_k = \Omega(b_k)$ with $b_k\to 0 $, then $a'_k=\Omega(b_k)$ too so
$$ \log( r_{k+1})  \leq \sum_{\ell=0}^k \log\left( \frac{1}{1+  \frac{a'_\ell}{C^2}  } \right)  + \log(r_{0}) \leq  -   \frac{1}{2C^2} \sum_{\ell=0}^k  a'_\ell +\log(r_{0}) \leq \left\{ \begin{array}{ll} - \frac{C'}{2C^2} k^{1-d} +\log(r_{0})  & \text{ case \emph{(b)} } \\  - \frac{C"}{2C^2} \log(k) +\log(r_{0})  & \text{ case \emph{(c)} }  \end{array}  \right. $$
leading, for $C_{b},C_{c}$ two positive constants, to 
$$ r_{k+1} \leq \left\{ \begin{array}{ll} C_{b}  \exp( -\frac{C'}{2C^2} k^{1-d})  & \text{ case \emph{(b)} }  \\ C_{c} \frac{1}{k^{C"/(2C^2)}}    & \text{ case \emph{(c)} }  \end{array}  \right. $$

\smallskip
\noindent\underline{Case $\theta\in]0,0.5[$.} Here $-2<2\theta-2<-1$ so as $0\leq r_{k+1} \leq r_k$, we have  $0\leq r_{k}^{2\theta-2} \leq r_{k+1}^{2\theta-2}$. 

Define $\phi(t) = \frac{C}{1-2\theta} t^{2\theta-1}$ with the same $C,\theta$ as in $\varphi$. Let us turn our attention to $\phi(r_{k+1}) - \phi(r_k)$ that we want to lower-bound by a constant. We proceed in two subcases:

\noindent If $r_{k+1}^{2\theta-2} \leq 2 r_{k}^{2\theta-2}$. Then we have
\begin{align*}
\phi(r_{k+1}) -  \phi(r_k)  &= -  \int^{r_k}_{r_{k+1}} \phi'(t) \mathrm{d}t  =  C  \int^{r_k}_{r_{k+1}} t^{2\theta-2} \mathrm{d}t \geq  C (r_{k} - r_{k+1}) r_{k}^{2\theta-2} \\
&\geq \frac{C}{2} (r_{k} - r_{k+1}) r_{k+1}^{2\theta-2}  = \frac{a_k}{2C} := d_{a_k}
\end{align*}

\noindent If $r_{k+1}^{2\theta-2} \geq 2 r_{k}^{2\theta-2}$. Then we have  $ r_{k+1}^{2\theta-1} \geq 2^{\frac{2\theta-1}{2\theta-2}} r_{k}^{2\theta-1}$
\begin{align*}
\phi(r_{k+1}) -  \phi(r_k)  &= \frac{C}{1-2\theta} (r_{k+1}^{2\theta-1}-r_k^{2\theta-1} ) \geq  \frac{C}{1-2\theta} \left( 2^{\frac{2\theta-1}{2\theta-2}} -1 \right)r_{k}^{2\theta-1} \geq  \frac{C}{1-2\theta} \left( 2^{\frac{2\theta-1}{2\theta-2}} -1 \right) r_{0}^{2\theta-1} := d_b
\end{align*}

Thus, we have $\phi(r_{k+1}) - \phi(r_k) \geq a'_k/(2C)$ with $a'_k = \min \{a_k,2C d_b\}$, thus 
$$ \phi(r_k) = \phi(r_{0}) + \sum_{\ell=0}^{k-1} \phi(r_{\ell+1})  -   \phi(r_{\ell})  \geq  \frac{1}{2C} \sum_{\ell=0}^{k-1} a'_\ell . $$
We have to split into two cases: \\
\noindent case \emph{(a)} $ \phi(r_k) \geq  k \min \{ a/2C ; d_b \}  $ hence the result holds by inverting $\phi$.\\
\noindent cases \emph{(b)} and \emph{(c)}, as before $a'_k = \Omega(b_k)$ and 
$$ 
\phi(r_{k+1}) \geq  \frac{1}{2C} \sum_{\ell=0}^k  a'_\ell \geq   \left\{ \begin{array}{ll}  C_b k^{1-d}    & \text{ case \emph{(b)} } \\[2ex] C_c \log(k)  & \text{ case \emph{(c)} }  \end{array}  \right.  
\qquad 
r_{k+1} \leq  \left\{ \begin{array}{ll}  \left( \frac{C}{(1-2\theta) C_b k^{1-d} } \right)^{\frac{1}{1-2\theta}}    & \text{ case \emph{(b)} } \\ \left( \frac{C}{(1-2\theta) C_c\log(k)} \right)^{\frac{1}{1-2\theta}}  & \text{ case \emph{(c)} }  \end{array}  \right.
$$
which leads to the claimed result.\qed
\end{proof}

The general convergence rates for weakly decreasing sequences thus depend on the geometry of the objective function (controlled by $C$ and $\theta$) and the speed of decrease of $\FF(x_k)$ is controlled by~$(a_k)$. To our knowledge, this is the first result accepting vanishing $(a_k)$. 
Note that the proof of our result uses techniques from \cite[Th.~2]{attouch2009convergence}; however, it deals with the convergence analysis of functional values and not the iterates sequence $(x_k)$, in contrast with the analysis of subgradient descent algorithms in \cite{attouch2009convergence,bolte2015error}. The latter analysis can cover the case of sequence with a stronger descent assumption on the algorithm (typically $\FF(x_{k+1}) \leq \FF(x_k) - a \|x_{k+1}-x_k\|^2$). Inertial versions of the proximal gradient algorithm, including the one studied here, do not satisfy this assumption.

\begin{rmk}[Global vs local KL assumptions]
In the above theorem and all next corollaries using the KL property \eqref{eq:klloc2}, we assume it holds for all $x\in \mathrm{dom}\,\FF$ for sake of simplicity (since the main topic of the paper is alternated inertia). In fact, the result holds under milder conditions; for instance, the exact same proof can be carried through without the global KL assumption in the following cases: 
\begin{itemize}
\item If \eqref{eq:klloc2} holds on $\{x: \FF(x)\leq \FF(x_0) \} $ only; then, as $\FF(x_k)\leq\FF(x_0)$, the property can be applied for the whole sequence $(x_k)$.
\item If $\FF$ is coercive, $(a_k)$ is non-increasing, and  \eqref{eq:klloc2} holds on $\{x: \FF(x)\leq \FF^\star + \eta \} $ for some $\eta>0$; then $(x_k)$ is bounded, 
and thus there is converging subsequence verifying $(\mathrm{dist}(0,\partial \FF(x_{k_n})))\to 0$ thus $\FF(x_k)\to\FF^\star$ monotonically so after some $K$, $\FF(x_k)\leq \FF^\star + \eta$  and \eqref{eq:basis2} holds.
\item If \eqref{eq:klloc2} holds on $\{x:\mathrm{dist}(0,\partial \FF(x))\leq \varepsilon\}\cup\{x: \FF(x)\leq \FF^\star + \eta \} $ for some $\varepsilon,\eta>0$; then the property can be used after some $K$ to retrieve convergence to $\FF^\star$ and  \eqref{eq:basis2} with the same arguments.
\end{itemize}
In addition, Theorem~\ref{th:ratekl2} also holds when \eqref{eq:klloc2} is verified for some local minimizer $\hat{\FF}$ instead of $\FF^\star$.
\end{rmk}

Combining the general convergence result with the derived descent of the proximal gradient algorithm with alternated inertia, we get the following convergence rates for our algorithm of interest.

\begin{coro}[Convergence and Rates of alternated inertial proximal gradient]
\label{co:alterkl}
Let $F$ verify the KL property \eqref{eq:klloc2}.
Then, for the alternated inertial proximal gradient algorithm \eqref{eq:altin}, $F(y_{2k})$ converges to $F^\star$ with the functional rates presented in Table~\ref{tab:rates} depending on the inertia parameters $(\alpha_k)$ taken in $[0,1]$, divided in three regimes:
\begin{itemize}
\item[(a)] when $\gamma<1/L$ and any inertia choice. Or when $\gamma=1/L$ and limited inertia $\alpha_k\leq \alpha<1$.
\item[(b)] when $\gamma=1/L$  and  inertia parameters $(\alpha_k)$ converging to $1$ at rate $1/k^d$ ($d\in]0,1[$).
\item[(c)] when $\gamma=1/L$ and inertia parameters $(\alpha_k)$ converging to $1$ at rate $1/k$.
\end{itemize}
\end{coro}

\begin{proof}
The result follows from the descent result of Theorem~\ref{th:main_des}, precisely Eq.~\eqref{eq:a_desKL} which gives 
$$  F( y_{k+2} )  \leq F(y_k)  - \frac{(2-\alpha_k-\gamma L)\gamma}{2(1+\gamma L)^2}  \dist(0,\partial F(y_{k+2} ))^2$$ 
combined with Theorem~\ref{th:ratekl2} applied with $(x_k) \equiv (y_{2k})$, $\Phi \equiv F$, and $a_k = \frac{(2-\alpha_k-\gamma L)\gamma}{2(1+\gamma L)^2}  $. \qed 
\end{proof}




The convergence rates are established for various inertial parameters $(\alpha_k)$ and strongly depends on the geometric properties: $\theta=1$ corresponding to finite convergence, $\theta\in[0.5,1[$  corresponding to linear or sublinear convergence, and  $\theta\in]0,0.5[$  corresponding to sublinear convergence. In particular, for case \emph{(b)} corresponding to the inertia type used in \cite{aujol2015stability}: in the least favorable case $\theta\in]0,0.5[$, convergence is at least $\mathcal{O}(1/k)$ when $d\leq2\theta$ and can be  $\mathcal{O}(1/k^2)$ for $\theta\in]0.25,0.5[$ with $d\leq 4\theta-1$. The case \emph{(c)} covers the popular Nesterov sequence \eqref{eq:nesta}.

Notice also that the present analysis does not extend easily to other monotonic accelerated proximal gradient algorithms. Indeed, MTwist \cite{bioucas2007new} and MFISTA \cite{beck2009fastb} only rely on function evaluation to get non-strict functional decrease and thus do not enjoy convergence rates from the above results. Concerning Monotone APG \cite{li2015accelerated}, the additional \emph{unaccelerated} step at each iteration makes it fit in the strongest case \emph{(a)} (the proof of which is inspired from \cite{li2015accelerated}) but at the cost of roughly doubling the computation expenses.

\section{Extensions}
\label{sec:ext}

\subsection{Resilience to Strong Convexity}
\label{sec:strong}




Let us consider the case where $F$ is in addition $\mu$-strongly convex (typically when $f$ is $\mu$-strongly convex and $g$ is simply convex) but this strong convexity is either undetected, local, or $\mu$ is unknown. Then, the vanilla proximal gradient algorithm is known to have exponential convergence (see \cite{karimi2016linear} for a proof based on error bounds) but accelerated versions such as FISTA only have polynomial convergence ($1/k^2$ in general). Moreover, these convergence rates clearly show in practice; see e.g. \cite{malitsky2016first} or the numerical illustrations of Section\;\ref{sec:num}. 
For the alternated inertial proximal gradient algorithm, we establish a better-than-polynomial convergence rate, which moreover translates in practice as shown in the numerical experiments.

\begin{prop}[Strongly convex case]
 Assume that $F$ is in addition $\mu$-strongly convex and take $\gamma$ and $(\alpha_k)$ as in Corollary~\ref{co:alterkl}. Then, the alternated inertial proximal gradient algorithm \eqref{eq:altin} verifies:
 \begin{itemize}
  \item \emph{better-than-polynomial rate:} provided that $\gamma$ and $(\alpha_k)$ are chosen as in \emph{(a)} or \emph{(b)}, $F(y_{2k})$ converges to $F^\star$ at rate $\mathcal{O}\left( \exp\left(-\nu k^{1-d}\right) \right)$ for $\nu>0$ and $d=0$ (case \emph{(a)}) or $d\in]0,1[$ (case \emph{(b)});
 \item \emph{iterates convergence:} $(y_{2k})$ converges to the unique minimizer of $F$ at the same rate.
 \end{itemize} 
\end{prop}

\begin{proof}
As $F$ is $\mu$-strongly convex, we have $F(x)-F^\star \geq \mu/2 \|x-x^\star\|^2 =  \mu/2 \dist^2(x,X^\star)$ which is a (global) error bound with perspective function $\varphi(t)=\sqrt{2/\mu}~t^{0.5}$. Thanks to the equivalence between error bounds and KL property for convex functions \cite[Th.~5]{bolte2015error}, we have for the alternated inertial proximal gradient the rates of Corollary~\ref{co:alterkl} with the perspective function $\tilde{\varphi}(t) = 2\sqrt{2/\mu}~ t^{0.5}$ i.e. $\theta=0.5$ (as the KL property actually holds for all $x$). So, provided that $(2-\alpha_k-\gamma L)$  does not vanish as quickly as $1/k$, we recover the rates of the top two lines of the column corresponding to $\theta=0.5$ in Theorem~\ref{th:ratekl2}. Finally, as $  \|y_{2k}-x^\star\|^2 \leq 2/\mu (F(y_{2k}) -F^\star) $, functional convergence (and rate) leads to iterates convergence (and rate). \qed
\end{proof}

\subsection{Case of Non-Convex $g$}\label{sec:noncvx}

In general, Kurdyka-{\L}ojasiewicz inequality is particularly suited to study non-convex case as a large number of non-convex functions verify KL properties. However, one has to find a functional decrease in algorithms to establish convergence properties. This decrease may be harder to obtain in the non-convex case. 
While one can get a functional decrease for the vanilla proximal gradient even when $f$ and $g$ are non-convex\footnote{In the case of the proximal gradient, the proof of Lemma~\ref{lem:1} recalled in the appendix requires (i) convexity of $f$ in order to take $x\neq y$ (see Eq.~\eqref{eq:cvxlem}); (ii) convexity of $g$ to get the term in $\|\mathsf{T}_\gamma(x) - y\|$ by strong convexity of the proximal surrogate (see Eq.~\eqref{eq:surr}).}, the presence of (alternated) inertia requires to be able to take $x\neq y$ in Lemma~\ref{lem:1} which prevents from getting rid of the convexity of $f$. 

For $f$ non-convex and $g$ convex, a solution, used for instance in the algorithm iPiano \cite{ochs2014ipiano}, is to slightly modify the inertial proximal gradient to recover a descent property on some Lyapunov function. Then, convergence and rates can be obtained by an adaptation of the existing results of \cite{Bolte2007KL}. 

Here, we rather study the case where the convexity of $g$ is dropped. The definition of the proximal operator directly extends \cite[Def.~1.22]{RocWet98} and our analysis indeed follows at the expense of a smaller range of stepsizes and inertia parameters, as formalized in the following proposition.

\begin{prop}[Non-convex case]
 Let Assumption~\ref{hyp:gen} hold without the convexity of $g$. Let $F$ verify the KL property \eqref{eq:klloc2}. Then, the alternated inertial proximal gradient algorithm \eqref{eq:altin} has the functional rates of Theorem~\ref{th:ratekl2} depending on $(\alpha_k)$ corresponding to the three regimes:
\begin{itemize}
\item[(a)] when $\gamma<1/(2L)$ and any $\alpha_k\leq1/2$. Or when $\gamma=1/(2L)$ and limited inertia $\alpha_k\leq \alpha<1/2$.
\item[(b)] when $\gamma=1/(2L)$ and inertia parameters $(\alpha_k)$ converging to $1/2$ at rate $1/k^d$ ($d\in]0,1[$).
\item[(c)] when $\gamma=1/(2L)$ and inertia parameters $(\alpha_k)$ converging to $1/2$ at rate $1/k$.
\end{itemize}
\end{prop}

\begin{proof}
We use Lemma~\ref{lem:1noncvx} which is the non-convex version of Lemma~\ref{lem:1}. The descent term in Theorem~\ref{th:main_des} is now factored with $(1-\alpha_k-\gamma L)$ instead of $(2-\alpha_k-\gamma L)$. Then, all the following results (and notably the rates) hold with this modification since the KL property is independent of convexity. \qed
\end{proof}

\subsection{Proximal Gradient Algorithm with Alternated Extrapolation}
\label{sec:altext}

In this subsection, we extend the alternated inertia to a more general alternated ``extrapolation'', the iterations of which take the form
\begin{align}
\label{eq:altex}
\left\{ \begin{array}{ll} y_{k+1} = \mathsf{T}_\gamma(x_k)  & ~~~ y_{k+2} = \mathsf{T}_\gamma(x_{k+1}) \\ x_{k+1} = \mathbf{extrapolation}\left( \{y_{\ell}\}_{\ell\leq k+1} \right) &  ~~~x_{k+2} = y_{k+2} \end{array} \right. 
\end{align}
where $\mathbf{extrapolation}\left( \{y_{\ell}\}_{\ell\leq k+1} \right)$ is a linear combination of past iterates (see for instance the recent \cite{liang2016multi}). We show here that the following extrapolation step, surprisingly close to \emph{heavy balls},
\begin{align}
\label{eq:alt1k2}
x_{k+1} = y_{k+1} -   \frac{1}{t_{k/2+1}}( y_{k+1} - y_k ) +  \frac{t_{k/2}-1}{t_{k/2+1}}( y_{k} - y_{k-1} ) \qquad \text{with $(t_k)$ as in \eqref{eq:nesta}} 
\end{align}
guarantees a worst case $\mathcal{O}(1/k^2)$ rate. In other words, alternating extrapolation allows attaining the same rates as standard inertia (e.g. FISTA) where alternating inertia cannot (see Corollary~\ref{co:alterkl}). Furthermore, if the KL property is additionally assumed then this rate may be improved to a small-o convergence instead of big-O for $\theta=0.5$ (with different arguments than in \cite{attouch2016rate}) and faster convergence as $\theta$ increases.

\begin{prop}[alternated extrapolation]
Take $\gamma\in]0,1/L]$. Then, the algorithm of \eqref{eq:altex} with the extrapolation defined in \eqref{eq:alt1k2} verifies for $k$ odd
\begin{align*}
     F(y_k) - F^\star \leq \frac{\|x_0 - x^\star\|^2}{2\gamma t_{k/2}^2} = \mathcal{O}\left(\frac{1}{k^2}\right) ~~ \text{ and }~~    & \| \mathsf{T}_\gamma(y_{k}) - y_{k}\|^2 = o\left(\frac{1}{t_{k/2}^2}\right) =  o\left(\frac{1}{k^2}\right)  .
\end{align*}
Moreover, if $F$ verifies the KL property \eqref{eq:klloc2}; then, we have that for $k$ odd
$$ F(y_{k}) - F^\star = o\left( \frac{1}{k^{\frac{1}{1-\theta}}} \right) . $$
\end{prop}

\begin{proof}
From the descent lemma with extrapolation \cite{beck2009fastb} 
with $x=x_{k}$ for $k$ even, and defining $F_k = F(y_k)-F^\star$, we get\\
\resizebox{0.99\textwidth}{!}{ $  t^2_{k/2+1} F_{k+2} - t_{k/2}^2 F_{k+1}  \leq  - \frac{1}{2\gamma}  \left[  \left\| t_{k/2+1} y_{k+2} -  (t_{{k/2}+1}-1)y_{k+1} -y^\star  \right\|^2 - \left\| t_{{k/2}+1} x_{k+1} -  (t_{{k/2}+1}-1)y_{k+1} -y^\star  \right\|^2   \right] .$ }\\
As $y_{k+1} =  \mathsf{T}_\gamma(x_k) =  \mathsf{T}_\gamma(y_k)$, we have $F(y_{k+1})\leq F(y_k) - \frac{1}{2\gamma}\| \mathsf{T}_\gamma(y_k) - y_k\|^2$ from Lemma~\ref{lem:1}, thus {\small
\begin{align*} 
t^2_{{k/2}+1} F_{k+2} - t_{k/2}^2 F_{k}  &\leq  - \frac{1}{2\gamma}  \left[  \left\| t_{{k/2}+1} y_{k+2} -  (t_{{k/2}+1}-1)y_{{k}+1} -y^\star  \right\|^2 - \left\| t_{{k/2}+1} x_{k+1} -  (t_{{k/2}+1}-1)y_{k+1} -y^\star  \right\|^2   \right] \\
& ~~~~  - \frac{t_{k/2}^2 }{2\gamma}\| \mathsf{T}_\gamma(y_k) - y_k\|^2
\end{align*}}
and in order to have a right hand side of the form  $ \frac{1}{2\gamma}  \left[   \left\| u_{k+2}  \right\|^2 - \left\|  u_{k}  \right\|^2 \right] $, one can choose the extrapolation, that is $x_{k+1}$ so that 
$$\left\| t_{{k/2}+1} x_{k+1} -  (t_{{k/2}+1}-1)y_{k+1} -y^\star  \right\|^2 =  \left\| t_{{k/2}} y_{k} -  (t_{{k/2}}-1)y_{k-1} -y^\star  \right\|^2$$
which naturally leads to taking
\begin{align*}
&  t_{{k/2}+1} x_{k+1} -  (t_{{k/2}+1}-1)y_{k+1} =  t_{{k/2}} y_{k} -  (t_{{k/2}}-1)y_{k-1}  \\
\Leftrightarrow ~~~~~~ &   x_{k+1} = y_{k+1} -  \frac{1}{t_{{k/2}+1}}( y_{k+1} - y_k ) +   \frac{t_{k/2}-1}{t_{{k/2}+1}}( y_{k} - y_{k-1} ) .
\end{align*}
This choice of extrapolation leads by summing the above inequality for $k$ even to
$$ t_{k/2}^2 F_{k}  +  \frac{1}{2\gamma} \sum_{\ell=0}^{k/2-1} t_{\ell}^2 \| \mathsf{T}_\gamma(y_{2\ell}) - y_{2\ell}\|^2 \leq  \frac{1}{2\gamma} \|y_0 - y^\star\|^2 $$
thus, noticing that all terms above are non-negative:\\
--  $F_{k} \leq \frac{1}{2\gamma t_{k/2}^2}  \|y_0 - y^\star\|^2 $.\\
--  $\sum_{\ell=0}^{\infty} t_{\ell}^2 \| \mathsf{T}_\gamma(y_{2\ell}) - y_{2\ell}\|^2 < +\infty$ thus $  \| \mathsf{T}_\gamma(y_{2\ell}) - y_{2\ell}\|^2 = o(1/t_{\ell}^2)$.\\
Interestingly, using  Lemma~\ref{lem:klopt}, we get that for $k$ even, $\text{dist}^2(0,\partial F(y_{k+1}) )  \leq (L+1/\gamma)^2\|\mathsf{T}_\gamma(y_k) - y_k\|^2  = o\left(\frac{1}{k^2}\right)$.  If we also assume that $F$ verifies \eqref{eq:klloc2} for some $C>0$, $\theta\in]0,1[$; then, we have that for $k$ odd $ F(y_{k}) - F^\star = o\left( \frac{1}{k^{\frac{1}{1-\theta}}} \right) $.
\qed
\end{proof}

\section{Numerical Experiments}
\label{sec:num}

\input{simus.tex}

\section{Conclusion}

Standard algorithms with inertia break down monotonicity; we have proved that adding a proximal gradient step after each inertial step enables to recover monotonicity without extra assumptions on the stepsizes or on the inertial sequence. The resulting alternating inertial algorithm enjoys the accelerated behavior of inertial algorithms while keeping the good properties of vanilla proximal gradient algorithms (automatic rate adaptation for strongly convex objectives and generalization for non-convex objectives). The complexity analysis of this algorithm is also the occasion of revealing a general convergence result of weakly decreasing algorithms for minimizing sharp objective functions. 

Many extensions of this idea of intermittent inertia are possible; we studied in particular an alternated extrapolation algorithm for which we show that it converges faster as the geometry improves, going beyond the usual worst case rate.


\appendix

\section{Appendix with Known Results about the Proximal Gradient}
\label{apx:lemmas}

For the sake of completeness, we provide short and direct proofs of known lemmas recalled in Section\;\ref{sec:pre}.

\begin{proof}{\em of Lemma~\ref{lem:1}}
Let $x\in\mathbb{R}^n$; by definition, we have 
\begin{align*}
\mathsf{T}_\gamma(x) &= \argmin_w \left( \gamma g(w) + \frac{1}{2} \left\| w- \left(x - \gamma \nabla f(x) \right)   \right\|^2  \right) 
= \argmin_w \left( \underbrace{ f(x) +  g(w) +  \langle w-x  ;  \nabla f(x) \rangle +  \frac{1}{2\gamma} \left\| w- x  \right\|^2 }_{s_x(w)} \right)
\end{align*}
and, as it is defined as the minimizer of $\frac{1}{\gamma}$-strongly convex surrogate function $s_x$, we have for any $y\in\mathbb{R}^n$ that  $s_x ( \mathsf{T}_\gamma(x)) + \frac{1}{2\gamma} \| \mathsf{T}_\gamma(x) - y\|^2 \leq s_x (y)$ so
\begin{align}
\label{eq:surr}  f(x) + g(\mathsf{T}_\gamma(x) ) +   & \langle \mathsf{T}_\gamma(x) -x  ;  \nabla f(x) \rangle +  \frac{\left\|  \mathsf{T}_\gamma(x) - x  \right\|^2 }{2\gamma} +  \frac{\left\|  \mathsf{T}_\gamma(x) - y  \right\|^2 }{2\gamma} \leq f(x) + g(y) +  \langle y-x  ;  \nabla f(x) \rangle +  \frac{\left\| y- x  \right\|^2 }{2\gamma}.
\end{align}
Now we use (i) the descent lemma on $L$-smooth function $f$ (see \cite[Th.~18.15]{bauschke2011convex}) to show that
\begin{align}
\label{eq:Lsmth}
f( \mathsf{T}_\gamma(x) ) \leq  f(x)  +   \langle \mathsf{T}_\gamma(x) -x  ;  \nabla f(x) \rangle + \frac{L}{2} \left\|  \mathsf{T}_\gamma(x) - x  \right\|^2
\end{align}
and (ii) the convexity of $f$ to have 
\begin{align}
\label{eq:cvxlem}
 f(x) +  \langle y-x  ;  \nabla f(x) \rangle \leq f(y) . 
\end{align}

Using Eq.~\eqref{eq:Lsmth} on the left hand side of \eqref{eq:surr} and Eq.~\eqref{eq:cvxlem} on the right hand side, we get
\begin{align*}
\label{eq:final} f(\mathsf{T}_\gamma(x)) + g(\mathsf{T}_\gamma(x) ) +  &  \frac{ (1-\gamma L)  \left\|  \mathsf{T}_\gamma(x) - x  \right\|^2 }{2\gamma} +  \frac{\left\|  \mathsf{T}_\gamma(x) - y  \right\|^2 }{2\gamma} \leq f(y) + g(y) +    \frac{\left\| y- x  \right\|^2 }{2\gamma}.
\end{align*}\qed
\end{proof}

\begin{proof} {\em of Lemma~\ref{lem:klopt}}
Let $x\in\mathbb{R}^n$, and let $y =  \mathsf{T}_\gamma(x) \in \argmin_w \left( \gamma g(w) + \frac{1}{2} \left\| w- \left(x - \gamma \nabla f(x) \right)   \right\|^2  \right)$, then
\begin{align*}
0 \in \gamma \partial g(y) + y - x +   \gamma \nabla f(x)  ~~~~ \Leftrightarrow  ~~~~ 0 \in \nabla f(y)  + \partial g(y)  +  \nabla f(x) -  \nabla f(y) + \frac{1}{\gamma}(y-x)
\end{align*}
so $  \nabla f(y) -  \nabla f(x) + \frac{1}{\gamma}(x-y) \in \partial F(y)$, thus we have 
$ \dist(0,\partial F(y)) \leq  \|  \nabla f(y) -  \nabla f(x) + \frac{1}{\gamma}(x-y) \| \leq \left(L + \frac{1}{\gamma}\right) \|x-y\|$.
\qed
\end{proof}

\begin{lemma}
\label{lem:1noncvx}
Let Assumption~\ref{hyp:gen} hold but with $g$ possibly nonconvex, and take $\gamma>0$. Then, for any $x,y\in\mathbb{R}^n$, 
\begin{equation*}
F( \mathsf{T}_\gamma(x) ) + \frac{ (1-\gamma L)}{2\gamma}  \left\|  \mathsf{T}_\gamma(x) - x  \right\|^2    \leq F(y)  + \frac{1}{2\gamma} \left\| x - y \right\|^2.
\end{equation*}
\end{lemma}

\begin{proof}
Let $x\in\mathbb{R}^n$; by definition, we have, as in the proof of Lemma~\ref{lem:1},
\begin{align*}
\mathsf{T}_\gamma(x) &= \argmin_w \left( \gamma g(w) + \frac{1}{2} \left\| w- \left(x - \gamma \nabla f(x) \right)   \right\|^2  \right) 
= \argmin_w \left( \underbrace{ f(x) +  g(w) +  \langle w-x  ;  \nabla f(x) \rangle +  \frac{1}{2\gamma} \left\| w- x  \right\|^2 }_{s_x(w)} \right)
\end{align*}
and, as it is defined as a minimizer of (non necessarily convex) surrogate function $s_x$, we have for any $y\in\mathbb{R}^n$ that  $s_x ( \mathsf{T}_\gamma(x)) \leq s_x (y)$ (which differs from from the convex case of  Lemma~\ref{lem:1}) thus
\begin{align}
\label{eq:surr2}
f(x) + g(\mathsf{T}_\gamma(x) ) +   & \langle \mathsf{T}_\gamma(x) -x  ;  \nabla f(x) \rangle +  \frac{\left\|  \mathsf{T}_\gamma(x) - x  \right\|^2 }{2\gamma}  \leq f(x) + g(y) +  \langle y-x  ;  \nabla f(x) \rangle +  \frac{\left\| y- x  \right\|^2 }{2\gamma}.
\end{align}
The proof then follows the same lines as that of Lemma~\ref{lem:1}.\qed

\end{proof}

\begin{acknowledgements}
The work of the authors is partly supported by the PGMO Grant \emph{Advanced Non-smooth Optimization Methods for Stochastic Programming}.
\end{acknowledgements}

\bibliographystyle{spmpsci_unsrt}
\bibliography{optim}

\end{document}

%% file: simus.tex
We compare the vanilla proximal gradient algorithm with
\begin{itemize}
\item \emph{proposed alternated inertial} version with $\alpha_k = t_{k+1}/(t_k-1)$, $t_k = ((k+a)/a)^d$, and $d=0.8$;
\item \emph{proposed alternated extrapolation} of Section~\ref{sec:altext};
\item \emph{inertial} versions: FISTA \cite{beck2009fast} and monotonic counterparts MFISTA \cite{beck2009fastb} and MAPG \cite{li2015accelerated} as well as lasso-specific MTwist \cite{bioucas2007new}.
\end{itemize}
All algorithms were tuned as advised in their respective sources, the initial point was drawn randomly and common to all methods. We will compare the speed of convergence of algorithms with respect to their number of iterations. One iteration corresponds to one proximal gradient evaluation; more precisely, the additional functional value computations for MFISTA, MTwist, and MAPG were not taken into account, and an iteration of MAPG is counted double as two proximal gradient steps are computed.

In Section~\ref{sec:num_log}, we consider the composite problem of $\ell_1$-regularized logisitic regression. We show that alternated inertia provides an efficient acceleration in the sense that it uniformly accelerates the proximal gradient even when large stepsizes are taken. Then, in Section~\ref{sec:num_lasso}, we address the lasso problem which illustrates the resilience of alternated inertia to strong convexity. Finally, in Section~\ref{sec:num_ncvx}, we revisit the two previous problems but replace the $\ell_1$-norm with the non-convex $\ell_{0.5}$-``\emph{norm}'', this allows us to illustrate the monotonicity and performance of alternated inertia even in a non-convex setting.

\subsection{$\ell_1$-Regularized Logistic Regression}
\label{sec:num_log}

We first consider the problem of $\ell_1$-regularized logistic regression on the popular \texttt{ionosphere} dataset\footnote{\url{https://archive.ics.uci.edu/ml/datasets/ionosphere}} containing $m=351$ binary classified examples $(a_i,y_i)\in\mathbb{R}^n\times\{-1,+1\}$ of feature size $n=35$. The problem writes as
$$ \min_{x\in\mathbb{R}^n } \underbrace{ \frac{1}{m} \sum_{i=1}^m \log\left(1+\exp(-y_i \langle a_i ;  x \rangle )\right)}_{f(x) } +  \underbrace{ \lambda_1 \|x\|_1 }_{g(x)}$$
and we fix the regularization parameter $\lambda_1$ to $0.1$.

In Figure~\ref{fig:log}, we plot the functional error for all compared algorithms with four choices for stepsize $\gamma$: (a) $\gamma =1/L_u$ where $L_u $ is the usual upper bound for the Lipschitz constant of the gradient of $f$ which is known to be overly pessimistic in most practical cases; and (b,c,d) $\gamma = \gamma_{\max}/\nu$ for $\nu=\{8,3,1.5\}$ with $\gamma_{\max}$ being the greatest stepsize admissible for the proximal gradient algorithm before divergence. This second set of stepsize enables to exhibit the behavior of the compared algorithm with practical and performing stepsizes. We notice that the alternated inertial version exhibits a steady monotonic behavior always outperforming the vanilla proximal gradient. Also, while for smaller stepsizes non-monotonic accelerated versions outperform monotonic ones, the trend shifts for the greater ones. Ultimately, for the greatest stepsize vanilla proximal gradient and alternated inertial counterpart are the most performing. Finally, one can see that the alternated extrapolation performs a bit worse than FISTA with the same characteristic oscillation with a frequency twice as low.

\begin{figure*}[ht!]
\input{Fig/fig_log.tex}
\caption{Comparison of the algorithms on $\ell_1$-regularized logistic regression (functional decrease vs number of iterations). The four cases correspond to four different choices of stepsize $\gamma$.\label{fig:log}}
\end{figure*}
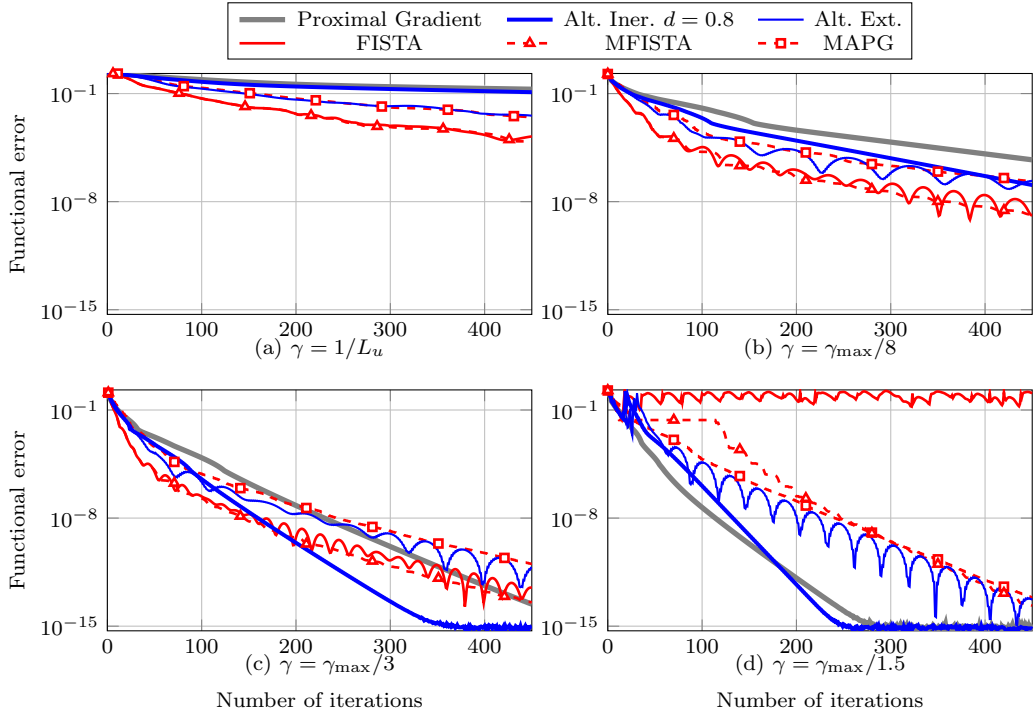

\subsection{$\ell_1$-Regularized Least-Squares (Lasso)}
\label{sec:num_lasso}

We address the lasso problem 
$$ \min_{x\in\mathbb{R}^n } \underbrace{ \left\|Ax-b\right\|_2^2}_{f(x) } +  \underbrace{ \lambda_1 \|x\|_1 }_{g(x)}$$
on synthetic matrix/vector couples $A\in\mathbb{R}^{m\times n}$ and $b\in \mathbb{R}^m$. $A$ is drawn from the standard normal distribution and $b = A x_0 + e$ where $x_0$ is a $10$\% sparse vector taken from the normal distribution, and $e$ is taken from the normal distribution with standard deviation $0.001$. We set $\lambda_1$ so that the original sparsity is ultimately recovered. 

The interest of this problem is that we can compute exactly the standard stepsize $\gamma=1/L$ where $L=2\|A^{\mathrm{T}}A\|$, which we use for all algorithms.
We plot the functional error for all compared algorithms with two different sizes of the matrix $A$, which correspond to two conditioning levels of the lasso  problem: (a) $130\times 80$ and (b)  $85\times 80 $.

In Figure~\ref{fig:lasso}, we observe that the proximal gradient and the alternated inertial counterpart benefit from a more-than-polynomial rate which enables them to outperform most other methods in the better conditioned case, the alternated inertial version being always significantly better than the vanilla one. In the second case, FISTA, MFISTA, and MTwist perform best although alternated inertia is still competitive.


\begin{figure*}[ht!]
\input{Fig/fig_lasso.tex}
\caption{Comparison of the algorithms on lasso (functional decrease vs number of iterations). The two cases correspond to two problems with two different strong convexity modulus (small and large).\label{fig:lasso}}
\end{figure*}
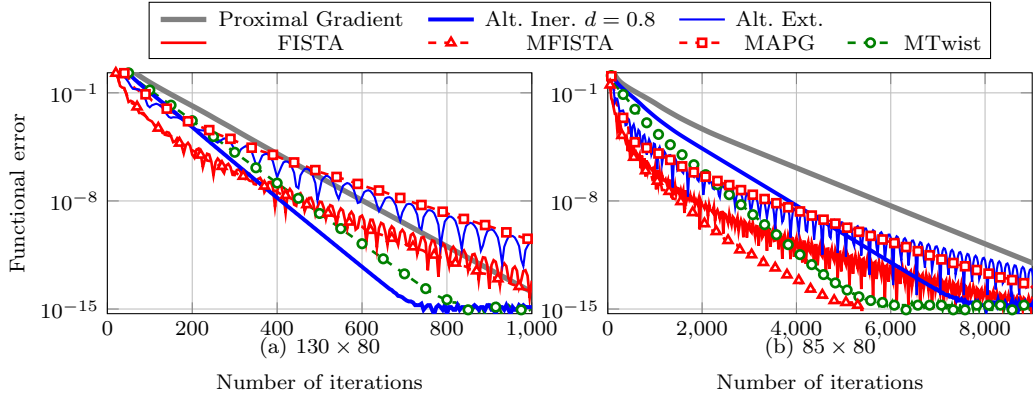

We also checked the wall clock times obtained by IPython's \texttt{timeit} in the situation of Fig.~\ref{fig:lasso} (a) and found that alternated inertia and FISTA roughly take the same the time as the proximal gradient while test-based MFISTA, MAPG, and MTwist are slightly more time consuming due to function evaluations; thus the presented comparisons are more than fair for alternated inertia.


\subsection{Non-Convex Problems}
\label{sec:num_ncvx}

In order to investigate the behavior of the compared algorithms in a non-convex setting, we replace the $\ell_1$ norm in the two previous problems by the $\ell_{0.5}$ ``\emph{norm}''. This regularization promotes sparsity in a stronger, non-convex, sense than the $\ell_1$ norm. This function writes $\|x\|_{0.5}^{0.5} = \sum_{i=1}^n \sqrt{|x_i|}$ and has the attractive feature of having a closed form proximal operator (see \cite{chartrand2016nonconvex} and references therein). For these non-convex problems, we do not plot the functional error, as the algorithms may reach different local minimizers depending on the initialization point, but rather: (a) the functional error with respect to the minimal value reached by the algorithm on the run; and (b) the distance of the subgradient to the null vector $\text{dist}^2(0,\partial F(x_k))$.
The two figures show that alternated inertia provides a quick and monotonic convergence. 



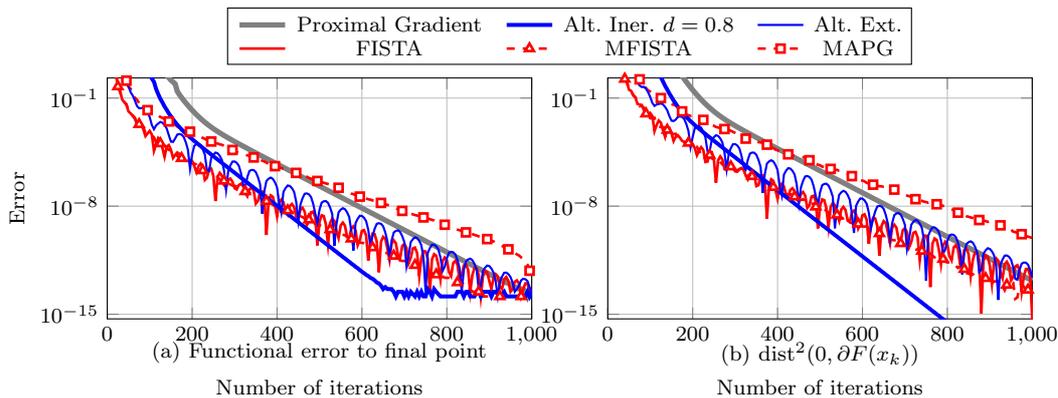
\begin{figure*}[ht!]
\input{Fig/fig_log_ncvx.tex}
\caption{Comparison of the algorithms on the $\ell_{0.5}$-regularized logistic regression problem (with $\lambda_{0.5} = 0.002$). \label{fig:ncvx1} }
\end{figure*}

\begin{figure*}[ht!]
\input{Fig/fig_lasso_ncvx.tex}
\caption{Comparison of the algorithms on the $\ell_{0.5}$-regularized linear regression problem (``$\ell_{0.5}$-lasso") with $m=130$ and $n=80$. To match the sparsity of $x_0$, 
$\lambda_{0.5}$ was set accordingly to $0.05$. \label{fig:ncvx2} }
\end{figure*}
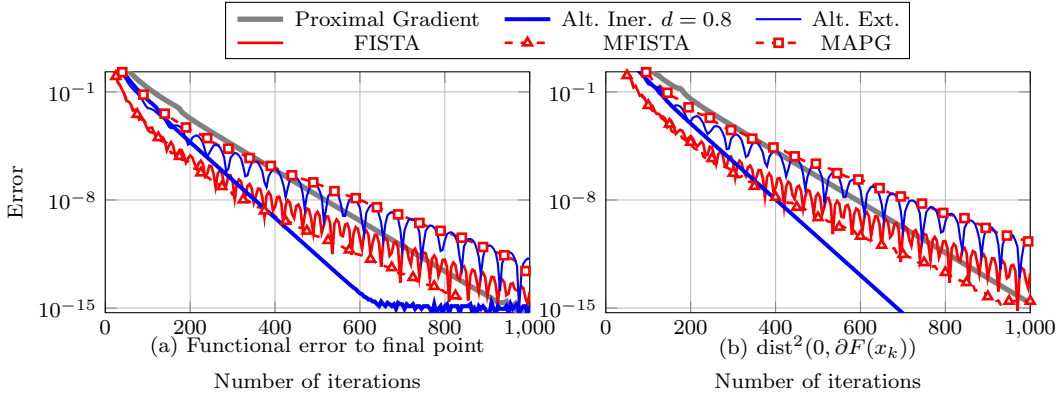

%% file: Fig/fig_log.tex
\begin{tikzpicture}[scale=1]
\begin{groupplot}[group style={group name=plot1, group size= 2 by 2},width=\textwidth]

\nextgroupplot[ 
 width=0.45\columnwidth, 
 height=0.3\columnwidth, 
 xmin=0, 
 xmax=450, 
 xmajorgrids, 
 ymode=log,
 ymin=5e-16, 
 ymax=2, 
 yminorticks=true, 
 ymajorgrids, 
 yminorgrids, 
 ylabel={Functional error }
 ]

 \addplot [ 
 color=gray, 
 solid, 
 line width=2.0pt, 
 mark size=2.5pt, 
 mark=none
 ] 
file { data/log4_PG.dat };
 \label{plots1:PG}

 \addplot [ 
 color=red, 
 dashed, 
 line width=1.0pt, 
 mark size=1.5pt, 
 mark=square*,
 mark options={solid,fill=white,draw=red}, 
 mark repeat={70}, 
 ]  
file { data/log4_MAPG.dat };
 \label{plots1:MAPG}

 \addplot [ 
 color=red, 
 dashed, 
 line width=1.0pt, 
 mark size=1.9pt, 
 mark=triangle*,
 mark options={solid,fill=white,draw=red}, 
 mark repeat={70}, 
 ] 
file { data/log4_MFISTA.dat };
 \label{plots1:MFISTA}

 \addplot [ 
 color=red, 
 solid, 
 line width=1.0pt, 
 mark size=1.8pt, 
 mark=none,
 mark options={solid,fill=white,draw=red}, 
 mark repeat={70}, 
 ] 
file { data/log4_FISTA.dat };
 \label{plots1:FISTA}

 \addplot [ 
 color=blue, 
 solid, 
 line width=1.5pt, 
 mark size=1.8pt, 
 mark=none,
 mark options={solid,fill=white,draw=red}, 
 mark repeat={70}, 
 ] 
file { data/log4_AI.dat };
 \label{plots1:AI}

 \addplot [ 
 color=blue, 
 solid, 
 line width=0.8pt, 
 mark size=1.8pt, 
 mark=none,
 mark options={solid,fill=white,draw=red}, 
 mark repeat={70}, 
 ]  
file { data/log4_AI2.dat };
 \label{plots1:AI2}

\nextgroupplot[ 
 width=0.45\columnwidth, 
 height=0.3\columnwidth, 
 xmin=0, 
 xmax=450, 
 xmajorgrids, 
 ymode=log,
 ymin=5e-16, 
 ymax=2, 
 yminorticks=true, 
 ymajorgrids, 
 yminorgrids, 
 ]

 \addplot [ 
 color=gray, 
 solid, 
 line width=2.0pt, 
 mark size=2.5pt, 
 mark=none
 ] 
file { data/log3_PG.dat };

 \addplot [ 
 color=red, 
 dashed, 
 line width=1.0pt, 
 mark size=1.5pt, 
 mark=square*,
 mark options={solid,fill=white,draw=red}, 
 mark repeat={70}, 
 ]  
file { data/log3_MAPG.dat };

 \addplot [ 
 color=red, 
 dashed, 
 line width=1.0pt, 
 mark size=1.9pt, 
 mark=triangle*,
 mark options={solid,fill=white,draw=red}, 
 mark repeat={70}, 
 ] 
file { data/log3_MFISTA.dat };

 \addplot [ 
 color=red, 
 solid, 
 line width=1.0pt, 
 mark size=1.8pt, 
 mark=none,
 mark options={solid,fill=white,draw=red}, 
 mark repeat={70}, 
 ] 
file { data/log3_FISTA.dat };

 \addplot [ 
 color=blue, 
 solid, 
 line width=1.5pt, 
 mark size=1.8pt, 
 mark=none,
 mark options={solid,fill=white,draw=red}, 
 mark repeat={70}, 
 ] 
file { data/log3_AI.dat };

 \addplot [ 
 color=blue, 
 solid, 
 line width=0.8pt, 
 mark size=1.8pt, 
 mark=none,
 mark options={solid,fill=white,draw=red}, 
 mark repeat={70}, 
 ]  
file { data/log3_AI2.dat };

\nextgroupplot[ 
 width=0.45\columnwidth, 
 height=0.3\columnwidth, 
 xmin=0, 
 xmax=450, 
 xmajorgrids, 
 xlabel={Number of iterations},
 x label style={at={(axis description cs:0.5,-0.22)},anchor=north},
 ymode=log,
 ymin=5e-16, 
 ymax=2, 
 yminorticks=true, 
 ymajorgrids, 
 yminorgrids, 
 ylabel={Functional error }
 ]

 \addplot [ 
 color=gray, 
 solid, 
 line width=2.0pt, 
 mark size=2.5pt, 
 mark=none
 ] 
file { data/log2_PG.dat };

 \addplot [ 
 color=red, 
 dashed, 
 line width=1.0pt, 
 mark size=1.5pt, 
 mark=square*,
 mark options={solid,fill=white,draw=red}, 
 mark repeat={70}, 
 ]  
file { data/log2_MAPG.dat };

 \addplot [ 
 color=red, 
 dashed, 
 line width=1.0pt, 
 mark size=1.9pt, 
 mark=triangle*,
 mark options={solid,fill=white,draw=red}, 
 mark repeat={70}, 
 ] 
file { data/log2_MFISTA.dat };

 \addplot [ 
 color=red, 
 solid, 
 line width=1.0pt, 
 mark size=1.8pt, 
 mark=none,
 mark options={solid,fill=white,draw=red}, 
 mark repeat={70}, 
 ] 
file { data/log2_FISTA.dat };

 \addplot [ 
 color=blue, 
 solid, 
 line width=1.5pt, 
 mark size=1.8pt, 
 mark=none,
 mark options={solid,fill=white,draw=red}, 
 mark repeat={70}, 
 ] 
file { data/log2_AI.dat };

 \addplot [ 
 color=blue, 
 solid, 
 line width=0.8pt, 
 mark size=1.8pt, 
 mark=none,
 mark options={solid,fill=white,draw=red}, 
 mark repeat={70}, 
 ]  
file { data/log2_AI2.dat };

\nextgroupplot[ 
 width=0.45\columnwidth, 
 height=0.3\columnwidth, 
 xmin=0, 
 xmax=450, 
 xmajorgrids, 
 xlabel={Number of iterations},
 x label style={at={(axis description cs:0.5,-0.22)},anchor=north},
 ymode=log,
 ymin=5e-16, 
 ymax=2, 
 yminorticks=true, 
 ymajorgrids, 
 yminorgrids, 
 ]

 \addplot [ 
 color=gray, 
 solid, 
 line width=2.0pt, 
 mark size=2.5pt, 
 mark=none
 ] 
file { data/log1_PG.dat };

 \addplot [ 
 color=red, 
 dashed, 
 line width=1.0pt, 
 mark size=1.5pt, 
 mark=square*,
 mark options={solid,fill=white,draw=red}, 
 mark repeat={70}, 
 ]  
file { data/log1_MAPG.dat };

 \addplot [ 
 color=red, 
 dashed, 
 line width=1.0pt, 
 mark size=1.9pt, 
 mark=triangle*,
 mark options={solid,fill=white,draw=red}, 
 mark repeat={70}, 
 ] 
file { data/log1_MFISTA.dat };

 \addplot [ 
 color=red, 
 solid, 
 line width=1.0pt, 
 mark size=1.8pt, 
 mark=none,
 mark options={solid,fill=white,draw=red}, 
 mark repeat={70}, 
 ] 
file { data/log1_FISTA.dat };

 \addplot [ 
 color=blue, 
 solid, 
 line width=1.5pt, 
 mark size=1.8pt, 
 mark=none,
 mark options={solid,fill=white,draw=red}, 
 mark repeat={70}, 
 ] 
file { data/log1_AI.dat };

 \addplot [ 
 color=blue, 
 solid, 
 line width=0.8pt, 
 mark size=1.8pt, 
 mark=none,
 mark options={solid,fill=white,draw=red}, 
 mark repeat={70}, 
 ]  
file { data/log1_AI2.dat };

\end{groupplot}

\node[anchor=north] at ($(plot1 c1r1.south)+(0,-0.2)$) { (a) $\gamma = 1/L_u$ };
\node[anchor=north] at ($(plot1 c2r1.south)+(0,-0.2)$) { (b) $\gamma = \gamma_{\max}/8$ };
\node[anchor=north] at ($(plot1 c1r2.south)+(0,-0.2)$) { (c) $\gamma = \gamma_{\max}/3$ };
\node[anchor=north] at ($(plot1 c2r2.south)+(0,-0.2)$) { (d) $\gamma = \gamma_{\max}/1.5$ };

\path (plot1 c1r1.north west|-current bounding box.north)--
      coordinate(legendpos)
      (plot1 c2r1.north east|-current bounding box.north);
\matrix[
    matrix of nodes,
    anchor=south,
    draw,
    inner sep=0.2em,
    draw
  ]at([yshift=1ex]legendpos)
  {
    \ref{plots1:PG}& Proximal Gradient &[5pt]
\ref{plots1:AI}& Alt. Iner. $d=0.8$ &[5pt]
\ref{plots1:AI2}& Alt. Ext.  \\
\ref{plots1:FISTA}& FISTA &[5pt]
\ref{plots1:MFISTA}& MFISTA &[5pt]
\ref{plots1:MAPG}& MAPG \\
};
\end{tikzpicture}

%% file: Fig/fig_lasso.tex
\begin{tikzpicture}[scale=1]
\begin{groupplot}[group style={group name=plot2, group size= 2 by 1},width=\textwidth]

\nextgroupplot[ 
 width=0.45\columnwidth, 
 height=0.3\columnwidth, 
 xmin=0, 
 xmax=1000, 
 xmajorgrids, 
 ymode=log,
 ymin=5e-16, 
 ymax=2, 
 yminorticks=true, 
 ymajorgrids, 
 yminorgrids, 
 ylabel={Functional error },
 xlabel={Number of iterations},
 x label style={at={(axis description cs:0.5,-0.22)},anchor=north},
 ]

 \addplot [  each nth point=5, filter discard warning=false, unbounded coords=discard, 
 color=gray, 
 solid, 
 line width=2.0pt, 
 mark size=2.5pt, 
 mark=none
 ] 
file { data/lasso_80_130_PG.dat };
 \label{plots1:PG}

 \addplot [  each nth point=5, filter discard warning=false, unbounded coords=discard, 
 color=green!50!black, 
 dashed, 
 line width=1.0pt, 
 mark size=1.6pt, 
 mark=*,
 mark options={solid,fill=white,draw=green!50!black}, 
 mark repeat={10}, 
 ]  
file { data/lasso_80_130_MTWIST.dat };
 \label{plots1:MTWIST}

 \addplot [  each nth point=5, filter discard warning=false, unbounded coords=discard, 
 color=red, 
 dashed, 
 line width=1.0pt, 
 mark size=1.5pt, 
 mark=square*,
 mark options={solid,fill=white,draw=red}, 
 mark repeat={10}, 
 ]  
file { data/lasso_80_130_MAPG.dat };
 \label{plots1:MAPG}

 \addplot [  each nth point=5, filter discard warning=false, unbounded coords=discard, 
 color=red, 
 dashed, 
 line width=1.0pt, 
 mark size=1.9pt, 
 mark=triangle*,
 mark options={solid,fill=white,draw=red}, 
 mark repeat={10}, 
 ] 
file { data/lasso_80_130_MFISTA.dat };
 \label{plots1:MFISTA}

 \addplot [  each nth point=5, filter discard warning=false, unbounded coords=discard, 
 color=red, 
 solid, 
 line width=1.0pt, 
 mark size=1.8pt, 
 mark=none,
 mark options={solid,fill=white,draw=red}, 
 mark repeat={10}, 
 ] 
file { data/lasso_80_130_FISTA.dat };
 \label{plots1:FISTA}

 \addplot [  each nth point=5, filter discard warning=false, unbounded coords=discard, 
 color=blue, 
 solid, 
 line width=1.5pt, 
 mark size=1.8pt, 
 mark=none,
 mark options={solid,fill=white,draw=red}, 
 mark repeat={10}, 
 ] 
file { data/lasso_80_130_AI.dat };
 \label{plots1:AI}

 \addplot [  each nth point=5, filter discard warning=false, unbounded coords=discard, 
 color=blue, 
 solid, 
 line width=0.8pt, 
 mark size=1.8pt, 
 mark=none,
 mark options={solid,fill=white,draw=red}, 
 mark repeat={10}, 
 ]  
file { data/lasso_80_130_AI2.dat };
 \label{plots1:AI2}

\nextgroupplot[ 
 width=0.45\columnwidth, 
 height=0.3\columnwidth, 
 xmin=0, 
 xmax=9000, 
 xmajorgrids, 
 ymode=log,
 ymin=5e-16, 
 ymax=2, 
 yminorticks=true, 
 ymajorgrids, 
 yminorgrids, 
 xlabel={Number of iterations},
 x label style={at={(axis description cs:0.5,-0.22)},anchor=north},
 ]

 \addplot [  each nth point=25, filter discard warning=false, unbounded coords=discard, 
 color=gray, 
 solid, 
 line width=2.0pt, 
 mark size=2.5pt, 
 mark=none
 ] 
file { data/lasso_80_80_PG.dat };

 \addplot [  each nth point=25, filter discard warning=false, unbounded coords=discard, 
 color=green!50!black, 
 dashed, 
 line width=1.0pt, 
 mark size=1.6pt, 
 mark=*,
 mark options={solid,fill=white,draw=green!50!black}, 
 mark repeat={10}, 
 ]  
file { data/lasso_80_80_MTWIST.dat };

 \addplot [  each nth point=25, filter discard warning=false, unbounded coords=discard, 
 color=red, 
 dashed, 
 line width=1.0pt, 
 mark size=1.5pt, 
 mark=square*,
 mark options={solid,fill=white,draw=red}, 
 mark repeat={10}, 
 ]  
file { data/lasso_80_80_MAPG.dat };

 \addplot [  each nth point=25, filter discard warning=false, unbounded coords=discard, 
 color=red, 
 dashed, 
 line width=1.0pt, 
 mark size=1.9pt, 
 mark=triangle*,
 mark options={solid,fill=white,draw=red}, 
 mark repeat={10}, 
 ] 
file { data/lasso_80_80_MFISTA.dat };

 \addplot [  each nth point=25, filter discard warning=false, unbounded coords=discard, 
 color=red, 
 solid, 
 line width=1.0pt, 
 mark size=1.8pt, 
 mark=none,
 mark options={solid,fill=white,draw=red}, 
 mark repeat={10}, 
 ] 
file { data/lasso_80_80_FISTA.dat };

 \addplot [  each nth point=25, filter discard warning=false, unbounded coords=discard, 
 color=blue, 
 solid, 
 line width=1.5pt, 
 mark size=1.8pt, 
 mark=none,
 mark options={solid,fill=white,draw=red}, 
 mark repeat={10}, 
 ] 
file { data/lasso_80_80_AI.dat };

 \addplot [  each nth point=25, filter discard warning=false, unbounded coords=discard, 
 color=blue, 
 solid, 
 line width=0.8pt, 
 mark size=1.8pt, 
 mark=none,
 mark options={solid,fill=white,draw=red}, 
 mark repeat={10}, 
 ]  
file { data/lasso_80_80_AI2.dat };

\end{groupplot}

\node[anchor=north] at ($(plot2 c1r1.south)+(0,-0.2)$) { (a) $130 \times 80$ };
\node[anchor=north] at ($(plot2 c2r1.south)+(0,-0.2)$) { (b) $85 \times 80$ };

\path (plot2 c1r1.north west|-current bounding box.north)--
      coordinate(legendpos)
      (plot2 c2r1.north east|-current bounding box.north);
\matrix[
    matrix of nodes,
    anchor=south,
    draw,
    inner sep=0.2em,
    draw
  ]at([yshift=1ex]legendpos)
  {
    \ref{plots1:PG}& Proximal Gradient &[5pt]
\ref{plots1:AI}& Alt. Iner. $d=0.8$ &[5pt]
\ref{plots1:AI2}& Alt. Ext.  \\
\ref{plots1:FISTA}& FISTA &[5pt]
\ref{plots1:MFISTA}& MFISTA &[5pt]
\ref{plots1:MAPG}& MAPG &[5pt]
\ref{plots1:MTWIST}& MTwist\\
};
\end{tikzpicture}

%% file: Fig/fig_log_ncvx.tex
\begin{tikzpicture}[scale=1]
\begin{groupplot}[group style={group name=plot2, group size= 2 by 1},width=\textwidth]

\nextgroupplot[ 
 width=0.45\columnwidth, 
 height=0.3\columnwidth, 
 xmin=0, 
 xmax=1000, 
 xmajorgrids, 
 ymode=log,
 ymin=5e-16, 
 ymax=2, 
 yminorticks=true, 
 ymajorgrids, 
 yminorgrids, 
 ylabel={Error},
 xlabel={Number of iterations},
 x label style={at={(axis description cs:0.5,-0.22)},anchor=north},
 ]

 \addplot [  each nth point=5, filter discard warning=false, unbounded coords=discard, 
 color=gray, 
 solid, 
 line width=2.0pt, 
 mark size=2.5pt, 
 mark=none
 ] 
file { data/ncvx2F_PG.dat };
 \label{plots1:PG}

 \addplot [  each nth point=5, filter discard warning=false, unbounded coords=discard, 
 color=red, 
 dashed, 
 line width=1.0pt, 
 mark size=1.5pt, 
 mark=square*,
 mark options={solid,fill=white,draw=red}, 
 mark repeat={10}, 
 ]  
file { data/ncvx2F_MAPG.dat };
 \label{plots1:MAPG}

 \addplot [  each nth point=5, filter discard warning=false, unbounded coords=discard, 
 color=red, 
 dashed, 
 line width=1.0pt, 
 mark size=1.9pt, 
 mark=triangle*,
 mark options={solid,fill=white,draw=red}, 
 mark repeat={10}, 
 ] 
file { data/ncvx2F_MFISTA.dat };
 \label{plots1:MFISTA}

 \addplot [  each nth point=5, filter discard warning=false, unbounded coords=discard, 
 color=red, 
 solid, 
 line width=1.0pt, 
 mark size=1.8pt, 
 mark=none,
 mark options={solid,fill=white,draw=red}, 
 mark repeat={10}, 
 ] 
file { data/ncvx2F_FISTA.dat };
 \label{plots1:FISTA}

 \addplot [  each nth point=5, filter discard warning=false, unbounded coords=discard, 
 color=blue, 
 solid, 
 line width=1.5pt, 
 mark size=1.8pt, 
 mark=none,
 mark options={solid,fill=white,draw=red}, 
 mark repeat={10}, 
 ] 
file { data/ncvx2F_AI.dat };
 \label{plots1:AI}

 \addplot [  each nth point=5, filter discard warning=false, unbounded coords=discard, 
 color=blue, 
 solid, 
 line width=0.8pt, 
 mark size=1.8pt, 
 mark=none,
 mark options={solid,fill=white,draw=red}, 
 mark repeat={10}, 
 ]  
file { data/ncvx2F_AI2.dat };
 \label{plots1:AI2}

\nextgroupplot[ 
 width=0.45\columnwidth, 
 height=0.3\columnwidth, 
 xmin=0, 
 xmax=1000, 
 xmajorgrids, 
 ymode=log,
 ymin=5e-16, 
 ymax=2, 
 yminorticks=true, 
 ymajorgrids, 
 yminorgrids, 
 xlabel={Number of iterations},
 x label style={at={(axis description cs:0.5,-0.22)},anchor=north},
 ]

 \addplot [  each nth point=5, filter discard warning=false, unbounded coords=discard, 
 color=gray, 
 solid, 
 line width=2.0pt, 
 mark size=2.5pt, 
 mark=none
 ] 
file { data/ncvx2E_PG.dat };

 \addplot [  each nth point=5, filter discard warning=false, unbounded coords=discard, 
 color=red, 
 dashed, 
 line width=1.0pt, 
 mark size=1.5pt, 
 mark=square*,
 mark options={solid,fill=white,draw=red}, 
 mark repeat={10}, 
 ]  
file { data/ncvx2E_MAPG.dat };

 \addplot [  each nth point=5, filter discard warning=false, unbounded coords=discard, 
 color=red, 
 dashed, 
 line width=1.0pt, 
 mark size=1.9pt, 
 mark=triangle*,
 mark options={solid,fill=white,draw=red}, 
 mark repeat={10}, 
 ] 
file { data/ncvx2E_MFISTA.dat };

 \addplot [  each nth point=5, filter discard warning=false, unbounded coords=discard, 
 color=red, 
 solid, 
 line width=1.0pt, 
 mark size=1.8pt, 
 mark=none,
 mark options={solid,fill=white,draw=red}, 
 mark repeat={10}, 
 ] 
file { data/ncvx2E_FISTA.dat };

 \addplot [  each nth point=5, filter discard warning=false, unbounded coords=discard, 
 color=blue, 
 solid, 
 line width=1.5pt, 
 mark size=1.8pt, 
 mark=none,
 mark options={solid,fill=white,draw=red}, 
 mark repeat={10}, 
 ] 
file { data/ncvx2E_AI.dat };

 \addplot [  each nth point=5, filter discard warning=false, unbounded coords=discard, 
 color=blue, 
 solid, 
 line width=0.8pt, 
 mark size=1.8pt, 
 mark=none,
 mark options={solid,fill=white,draw=red}, 
 mark repeat={10}, 
 ]  
file { data/ncvx2E_AI2.dat };

\end{groupplot}

\node[anchor=north] at ($(plot2 c1r1.south)+(0,-0.2)$) { (a) Functional error to final point };
\node[anchor=north] at ($(plot2 c2r1.south)+(0,-0.2)$) { (b) $\text{dist}^2(0,\partial F(x_k))$ };

\path (plot2 c1r1.north west|-current bounding box.north)--
      coordinate(legendpos)
      (plot2 c2r1.north east|-current bounding box.north);
\matrix[
    matrix of nodes,
    anchor=south,
    draw,
    inner sep=0.2em,
    draw
  ]at([yshift=1ex]legendpos)
  {
    \ref{plots1:PG}& Proximal Gradient &[5pt]
\ref{plots1:AI}& Alt. Iner. $d=0.8$ &[5pt]
\ref{plots1:AI2}& Alt. Ext.  \\
\ref{plots1:FISTA}& FISTA &[5pt]
\ref{plots1:MFISTA}& MFISTA &[5pt]
\ref{plots1:MAPG}& MAPG \\
};
\end{tikzpicture}

%% file: Fig/fig_lasso_ncvx.tex
\begin{tikzpicture}[scale=1]
\begin{groupplot}[group style={group name=plot2, group size= 2 by 1},width=\textwidth]

\nextgroupplot[ 
 width=0.45\columnwidth, 
 height=0.3\columnwidth, 
 xmin=0, 
 xmax=1000, 
 xmajorgrids, 
 ymode=log,
 ymin=5e-16, 
 ymax=2, 
 yminorticks=true, 
 ymajorgrids, 
 yminorgrids, 
 ylabel={Error},
 xlabel={Number of iterations},
 x label style={at={(axis description cs:0.5,-0.22)},anchor=north},
 ]

 \addplot [  each nth point=5, filter discard warning=false, unbounded coords=discard, 
 color=gray, 
 solid, 
 line width=2.0pt, 
 mark size=2.5pt, 
 mark=none
 ] 
file { data/ncvxF_PG.dat };
 \label{plots1:PG}

 \addplot [  each nth point=5, filter discard warning=false, unbounded coords=discard, 
 color=red, 
 dashed, 
 line width=1.0pt, 
 mark size=1.5pt, 
 mark=square*,
 mark options={solid,fill=white,draw=red}, 
 mark repeat={10}, 
 ]  
file { data/ncvxF_MAPG.dat };
 \label{plots1:MAPG}

 \addplot [  each nth point=5, filter discard warning=false, unbounded coords=discard, 
 color=red, 
 dashed, 
 line width=1.0pt, 
 mark size=1.9pt, 
 mark=triangle*,
 mark options={solid,fill=white,draw=red}, 
 mark repeat={10}, 
 ] 
file { data/ncvxF_MFISTA.dat };
 \label{plots1:MFISTA}

 \addplot [  each nth point=5, filter discard warning=false, unbounded coords=discard, 
 color=red, 
 solid, 
 line width=1.0pt, 
 mark size=1.8pt, 
 mark=none,
 mark options={solid,fill=white,draw=red}, 
 mark repeat={10}, 
 ] 
file { data/ncvxF_FISTA.dat };
 \label{plots1:FISTA}

 \addplot [  each nth point=5, filter discard warning=false, unbounded coords=discard, 
 color=blue, 
 solid, 
 line width=1.5pt, 
 mark size=1.8pt, 
 mark=none,
 mark options={solid,fill=white,draw=red}, 
 mark repeat={10}, 
 ] 
file { data/ncvxF_AI.dat };
 \label{plots1:AI}

 \addplot [  each nth point=5, filter discard warning=false, unbounded coords=discard, 
 color=blue, 
 solid, 
 line width=0.8pt, 
 mark size=1.8pt, 
 mark=none,
 mark options={solid,fill=white,draw=red}, 
 mark repeat={10}, 
 ]  
file { data/ncvxF_AI2.dat };
 \label{plots1:AI2}

\nextgroupplot[ 
 width=0.45\columnwidth, 
 height=0.3\columnwidth, 
 xmin=0, 
 xmax=1000, 
 xmajorgrids, 
 ymode=log,
 ymin=5e-16, 
 ymax=2, 
 yminorticks=true, 
 ymajorgrids, 
 yminorgrids, 
 xlabel={Number of iterations},
 x label style={at={(axis description cs:0.5,-0.22)},anchor=north},
 ]

 \addplot [  each nth point=5, filter discard warning=false, unbounded coords=discard, 
 color=gray, 
 solid, 
 line width=2.0pt, 
 mark size=2.5pt, 
 mark=none
 ] 
file { data/ncvxE_PG.dat };

 \addplot [  each nth point=5, filter discard warning=false, unbounded coords=discard, 
 color=red, 
 dashed, 
 line width=1.0pt, 
 mark size=1.5pt, 
 mark=square*,
 mark options={solid,fill=white,draw=red}, 
 mark repeat={10}, 
 ]  
file { data/ncvxE_MAPG.dat };

 \addplot [  each nth point=5, filter discard warning=false, unbounded coords=discard, 
 color=red, 
 dashed, 
 line width=1.0pt, 
 mark size=1.9pt, 
 mark=triangle*,
 mark options={solid,fill=white,draw=red}, 
 mark repeat={10}, 
 ] 
file { data/ncvxE_MFISTA.dat };

 \addplot [  each nth point=5, filter discard warning=false, unbounded coords=discard, 
 color=red, 
 solid, 
 line width=1.0pt, 
 mark size=1.8pt, 
 mark=none,
 mark options={solid,fill=white,draw=red}, 
 mark repeat={10}, 
 ] 
file { data/ncvxE_FISTA.dat };

 \addplot [  each nth point=5, filter discard warning=false, unbounded coords=discard, 
 color=blue, 
 solid, 
 line width=1.5pt, 
 mark size=1.8pt, 
 mark=none,
 mark options={solid,fill=white,draw=red}, 
 mark repeat={10}, 
 ] 
file { data/ncvxE_AI.dat };

 \addplot [  each nth point=5, filter discard warning=false, unbounded coords=discard, 
 color=blue, 
 solid, 
 line width=0.8pt, 
 mark size=1.8pt, 
 mark=none,
 mark options={solid,fill=white,draw=red}, 
 mark repeat={10}, 
 ]  
file { data/ncvxE_AI2.dat };

\end{groupplot}

\node[anchor=north] at ($(plot2 c1r1.south)+(0,-0.2)$) { (a) Functional error to final point };
\node[anchor=north] at ($(plot2 c2r1.south)+(0,-0.2)$) { (b) $\text{dist}^2(0,\partial F(x_k))$ };

\path (plot2 c1r1.north west|-current bounding box.north)--
      coordinate(legendpos)
      (plot2 c2r1.north east|-current bounding box.north);
\matrix[
    matrix of nodes,
    anchor=south,
    draw,
    inner sep=0.2em,
    draw
  ]at([yshift=1ex]legendpos)
  {
    \ref{plots1:PG}& Proximal Gradient &[5pt]
\ref{plots1:AI}& Alt. Iner. $d=0.8$ &[5pt]
\ref{plots1:AI2}& Alt. Ext.  \\
\ref{plots1:FISTA}& FISTA &[5pt]
\ref{plots1:MFISTA}& MFISTA &[5pt]
\ref{plots1:MAPG}& MAPG \\
};
\end{tikzpicture}

%% file: AI.bbl
\def\cprime{$'$} \def\cdprime{$''$} \def\cprime{$'$} \def\cprime{$'$}
\begin{thebibliography}{10}
\providecommand{\url}[1]{{#1}}
\providecommand{\urlprefix}{URL }
\expandafter\ifx\csname urlstyle\endcsname\relax
  \providecommand{\doi}[1]{DOI~\discretionary{}{}{}#1}\else
  \providecommand{\doi}{DOI~\discretionary{}{}{}\begingroup
  \urlstyle{rm}\Url}\fi

\bibitem{chambolle1998nonlinear}
Chambolle, A., De~Vore, R.A., Lee, N.Y., Lucier, B.J.: Nonlinear wavelet image
  processing: variational problems, compression, and noise removal through
  wavelet shrinkage.
\newblock IEEE Transactions on Image Processing \textbf{7}(3), 319--335 (1998)

\bibitem{daubechies2004iterative}
Daubechies, I., Defrise, M., De~Mol, C.: An iterative thresholding algorithm
  for linear inverse problems with a sparsity constraint.
\newblock Communications on pure and applied mathematics \textbf{57}(11),
  1413--1457 (2004)

\bibitem{hale2008fixed}
Hale, E.T., Yin, W., Zhang, Y.: Fixed-point continuation for
  $\ell_1$-minimization: Methodology and convergence.
\newblock SIAM Journal on Optimization \textbf{19}(3), 1107--1130 (2008)

\bibitem{alvarez2004weak}
Alvarez, F.: Weak convergence of a relaxed and inertial hybrid
  projection-proximal point algorithm for maximal monotone operators in hilbert
  space.
\newblock SIAM Journal on Optimization \textbf{14}(3), 773--782 (2004)

\bibitem{lorenz2014inertial}
Lorenz, D., Pock, T.: An inertial forward-backward algorithm for monotone
  inclusions.
\newblock Journal of Mathematical Imaging and Vision \textbf{51}(2), 311--325
  (2014)

\bibitem{chambolle2015convergence}
Chambolle, A., Dossal, C.: On the convergence of the iterates of the “fast
  iterative shrinkage/thresholding algorithm”.
\newblock Journal of Optimization Theory and Applications \textbf{166}(3),
  968--982 (2015)

\bibitem{aujol2015stability}
Aujol, J.F., Dossal, C.: Stability of over-relaxations for the forward-backward
  algorithm, application to fista.
\newblock SIAM Journal on Optimization \textbf{25}(4), 2408--2433 (2015)

\bibitem{attouch2016rate}
Attouch, H., Peypouquet, J.: The rate of convergence of {Nesterov's}
  accelerated forward-backward method is actually faster than $1/k^2$.
\newblock SIAM Journal on Optimization \textbf{26}(3), 1824--1834 (2016)

\bibitem{nesterov1983method}
Nesterov, Y.: A method of solving a convex programming problem with convergence
  rate o (1/k2).
\newblock Soviet Mathematics Doklady \textbf{27}(2), 372--376 (1983)

\bibitem{guler1992new}
G{\"u}ler, O.: New proximal point algorithms for convex minimization.
\newblock SIAM Journal on Optimization \textbf{2}(4), 649--664 (1992)

\bibitem{beck2009fast}
Beck, A., Teboulle, M.: A fast iterative shrinkage-thresholding algorithm for
  linear inverse problems.
\newblock SIAM journal on imaging sciences \textbf{2}(1), 183--202 (2009)

\bibitem{nesterov2005smooth}
Nesterov, Y.: Smooth minimization of non-smooth functions.
\newblock Mathematical programming \textbf{103}(1), 127--152 (2005)

\bibitem{mainge2008convergence}
Maing{\'e}, P.E.: Convergence theorems for inertial km-type algorithms.
\newblock Journal of Computational and Applied Mathematics \textbf{219}(1),
  223--236 (2008)

\bibitem{beck2009fastb}
Beck, A., Teboulle, M.: Fast gradient-based algorithms for constrained total
  variation image denoising and deblurring problems.
\newblock IEEE Transactions on Image Processing \textbf{18}(11), 2419--2434
  (2009)

\bibitem{malitsky2016first}
Malitsky, Y., Pock, T.: A first-order primal-dual algorithm with linesearch.
\newblock arXiv preprint arXiv:1608.08883  (2016)

\bibitem{correa1993}
Correa, R., {Lemar\'echal}, C.: Convergence of some algorithms for convex
  minimization.
\newblock Mathematical Programming \textbf{62}(2), 261--275 (1993)

\bibitem{bioucas2007new}
Bioucas-Dias, J.M., Figueiredo, M.A.: A new twist: two-step iterative
  shrinkage/thresholding algorithms for image restoration.
\newblock IEEE Transactions on Image processing \textbf{16}(12), 2992--3004
  (2007)

\bibitem{fuentes2012descentwise}
Fuentes, M., Malick, J., Lemar{\'e}chal, C.: Descentwise inexact proximal
  algorithms for smooth optimization.
\newblock Computational Optimization and Applications \textbf{53}(3), 755--769
  (2012)

\bibitem{li2015accelerated}
Li, H., Lin, Z.: Accelerated proximal gradient methods for nonconvex
  programming.
\newblock In: Advances in neural information processing systems, pp. 379--387
  (2015)

\bibitem{mu2015note}
Mu, Z., Peng, Y.: A note on the inertial proximal point method.
\newblock Statistics, Optimization \& Information Computing \textbf{3}(3),
  241--248 (2015)

\bibitem{iutzeler2016generic}
Iutzeler, F., Hendrickx, J.M.: A generic linear rate acceleration of
  optimization algorithms via relaxation and inertia.
\newblock arXiv preprint arXiv:1603.05398  (2016)

\bibitem{hiriart-lemarechal-1993}
Hiriart-Urruty, J.B., {Lemar\'echal}, C.: Convex Analysis and
  Min\-im\-iz\-a\-tion Alg\-or\-ithms.
\newblock Springer Verlag, Heidelberg (1993).
\newblock Two volumes

\bibitem{bauschke2011convex}
Bauschke, H.H., Combettes, P.L.: Convex analysis and monotone operator theory
  in Hilbert spaces.
\newblock Springer Science \& Business Media (2011)

\bibitem{bolte2015error}
Bolte, J., Nguyen, T.P., Peypouquet, J., Suter, B.W.: From error bounds to the
  complexity of first-order descent methods for convex functions.
\newblock Mathematical Programming pp. 1--37 (2015)

\bibitem{thesetrong}
NGuyen, T.P.: Kurdyka-lojasiewicz and convexity: algorithms and applications.
\newblock Ph.D. thesis, Toulouse University (2017)

\bibitem{RocWet98}
Rockafellar, R.T., Wets, R.J.B.: Variational Analysis.
\newblock Springer (1998)

\bibitem{Bolte2007KL}
Bolte, J., Daniilidis, A., Lewis, A.: The Łojasiewicz inequality for nonsmooth
  subanalytic functions with applications to subgradient dynamical systems.
\newblock SIAM Journal on Optimization \textbf{17}(4), 1205--1223 (2007)

\bibitem{attouch2009convergence}
Attouch, H., Bolte, J.: On the convergence of the proximal algorithm for
  nonsmooth functions involving analytic features.
\newblock Mathematical Programming \textbf{116}(1), 5--16 (2009)

\bibitem{frankel2015splitting}
Frankel, P., Garrigos, G., Peypouquet, J.: Splitting methods with variable
  metric for kurdyka--{\l}ojasiewicz functions and general convergence rates.
\newblock Journal of Optimization Theory and Applications \textbf{165}(3),
  874--900 (2015)

\bibitem{karimi2016linear}
Karimi, H., Nutini, J., Schmidt, M.: Linear convergence of gradient and
  proximal-gradient methods under the polyak-{\l}ojasiewicz condition.
\newblock In: Joint European Conference on Machine Learning and Knowledge
  Discovery in Databases, pp. 795--811 (2016)

\bibitem{ochs2014ipiano}
Ochs, P., Chen, Y., Brox, T., Pock, T.: ipiano: Inertial proximal algorithm for
  nonconvex optimization.
\newblock SIAM Journal on Imaging Sciences \textbf{7}(2), 1388--1419 (2014)

\bibitem{liang2016multi}
Liang, J., Fadili, J., Peyr{\'e}, G.: A multi-step inertial forward-backward
  splitting method for non-convex optimization.
\newblock In: Advances in Neural Information Processing Systems, pp. 4035--4043
  (2016)

\bibitem{chartrand2016nonconvex}
Chartrand, R., Yin, W.: Nonconvex sparse regularization and splitting
  algorithms.
\newblock In: Splitting Methods in Communication, Imaging, Science, and
  Engineering, pp. 237--249. Springer (2016)

\end{thebibliography}
